\documentclass[11pt]{amsart}
\usepackage[margin=1in]{geometry}

\usepackage{amsmath,amsfonts,amssymb,amsthm}
\usepackage{mathtools}
\usepackage{asymptote}
\usepackage{cancel}
\usepackage[colorlinks, backref=page, linktocpage = true, citecolor = purple, linkcolor = purple]{hyperref}
\usepackage{mathabx}

\usepackage{bm}
\newtheorem{thm}{Theorem}
\numberwithin{thm}{section}
\newtheorem{prop}[thm]{Proposition}
\newtheorem{lemma}[thm]{Lemma}
\newtheorem{cor}[thm]{Corollary}
\newtheorem{question}[thm]{Question}
\newtheorem*{thmg}{Theorem \ref{thm:graphs}}
\newtheorem*{thmnz}{Theorem \ref{thm:nonzero}}
\newtheorem*{thmm2}{Theorem \ref{thm:main2}}
\newtheorem*{lemmnf}{Lemma \ref{lem:multinomial-form}}

\newtheorem*{cor2}{Corollaries \ref{cor:lower-bound} and \ref{cor:upper-bound}}

\newcommand{\nQ}{NQ}

\theoremstyle{definition}
\newtheorem{defn}[thm]{Definition}
\newtheorem{remark}[thm]{Remark}
\newtheorem{example}[thm]{Example}
\newtheorem{notation}[thm]{Notation}

\newtheorem*{defnmc}{Definition \ref{defn:multicolored}}

\usepackage{todonotes}

\usepackage{youngtab}
\usepackage{ytableau}

\definecolor{mygreen}{rgb}{0.2, 0.6, 0.2}
\definecolor{y1}{rgb}{0.8, 0.6, 0}
\definecolor{y2}{rgb}{0.5, 0, 1}

\newcommand{\phihue}{\varphi^{\mathrm{hue}}}

\newcommand{\Mbar}{\overline{M}}
\newcommand{\Moduli}{\overline{M}} 
\newcommand{\moduli}{M} 

\newcommand{\Fix}{\mathrm{Fix}}
\newcommand{\OSP}{\mathrm{OSP}}

\newcommand{\Trees}{\mathrm{Trees}}
\newcommand{\DR}{\mathrm{DR}}
\newcommand{\invOSP}{\varphi^{\mathrm{OSP}}}
\newcommand{\invheavy}{\varphi^{\updownarrows}}
\newcommand{\tree}{\mathrm{Tree}}
\newcommand{\millipedes}{\mathrm{Mill}}

\newcommand{\PP}{\mathbb{P}}

\usepackage{xfrac}
\usepackage{graphicx} 

\newcommand{\ACO}{\mathrm{ACO}}

\title{Sign-reversing involutions in moduli spaces of curves}
\author{Vance Blankers}
\address{College of Science, Northeastern University, Boston, MA 02115, USA}
\email{v.blankers@northeastern.edu}

\author{Maria Gillespie} 
\address{Department of Mathematics, Colorado State University, Fort Collins, CO 80523, USA}
\email{maria.gillespie@colostate.edu}
\thanks{The second author was partially supported by the Simons Foundation.}

\author{Jake Levinson} 
\address{Département de mathématiques et de statistique, Université de Montréal, Montréal, QC, Canada} 
\email{jake.levinson@umontreal.ca}
\thanks{The third author was partially supported by NSERC Discovery Grant RGPIN-2021-04169.}

\date{\today}

\begin{document}

\begin{abstract}
 We use sign-reversing involutions to solve two computational problems that arise naturally in the geometry of moduli spaces of curves.  In particular, we give an explicit, nonnegative combinatorial formula for arbitrary $\psi$ class intersection products on the genus zero multicolored spaces $\Mbar_{0,[r_1,\ldots,r_m]}$ using a novel sign reversing involution on decorated diagrams.  As an application, we give a necessary and sufficient condition for when these intersection products are positive in terms of matchings on graphs. 
 
 We also calculate the analog of the tropical Euler characteristic for the graphical moduli spaces $\Mbar_{0,\Gamma}$ for graphs with two dominant vertices $P, Q$, by constructing two new sign-reversing involutions to simplify the sum. We show that (up to sign) it is the number of acyclic orientations of $\Gamma \smallsetminus  \{P, Q\}$.
\end{abstract}

\maketitle

\section{Introduction}\label{sec:introduction}

We investigate two natural alternating sums arising on modular compactifications of the moduli space $M_{0, n}$ of $n$ distinct points on $\mathbb{P}^1$, specifically Fry's \emph{graphically stable moduli spaces} $\overline{M}_{0, \Gamma}$ \cite{fry2019tropical} and the special case of what we call \emph{multicolored} moduli spaces, which arise naturally in tropical geometry. We examine
alternating sums of boundary strata and alternating sums computing intersection products of $\psi$ classes, evaluating these respective sums using the combinatorial method of \emph{sign-reversing involutions} (SRIs) to reduce them to positive expressions. We thereby obtain new enumerative and positivity results. 

\begin{remark}
This paper grew out of conversations between combinatorialists and geometers at the 2024 BIRS workshop \emph{Combinatorics of Moduli of Curves}, where sign-reversing involutions were identified as a relatively underused technique in this area, despite the prevalence of alternating sums in both algebraic geometry (Euler characteristics, determinants) and in combinatorics, the latter including summations over graphs and trees resembling those encountered in studying moduli spaces of curves. We thus intend our paper as a demonstration of the usefulness and flexibility of sign-reversing involutions, in particular as applied to moduli spaces of curves.
\end{remark}

The moduli space $M_{0, n}$ admits many natural modular compactifications, including the Deligne--Mumford compactification by stable curves, the Hassett spaces of weighted stable curves, and Fry's graphically stable moduli spaces $\Moduli_{0, \Gamma}$ \cite{fry2019tropical}, among others (see \cite{smyth_zstable} for a complete classification). These spaces differ in their boundary geometry, which in turn affects their enumerative and topological properties. In terms of marked curves, they differ in their rules governing when collisions of marked points result in new components. For example, the graphically-stable moduli space $\overline{M}_{0, \Gamma}$ is determined by a choice of simple non-bipartite graph $\Gamma$ on the vertex set $[n] := \{1, \ldots, n\}$. For $\Gamma$-stable curves, the marked points are then allowed to coincide whenever the corresponding vertices of $\Gamma$ form an independent set. The boundary strata of $\Mbar_{0,\Gamma}$ are indexed by the set $\tree(\Gamma)$ of $\Gamma$-stable trees (see Definition \ref{defn:dual_tree}). 

The following special cases are distinguished in tropical geometry:
\begin{defnmc}
    Let $\Gamma$ be a complete multipartite graph on $[n]$ with independent sets $C^{(1)}, \dots, C^{(m)}$ of respective sizes $r_1,\dots,r_m$, with $r_1 + \cdots r_m = n$. We call $\Mbar_{0,[r_1,\dots,r_m]} := \Mbar_{0,\Gamma}$ a \textbf{multicolored space}. Stability with respect to $\Gamma$ can be thought of here as each marked point having a color, where marked points are allowed to coincide if and only if they have the same color.
\end{defnmc}

Multicolored spaces have particularly nice combinatorial and topological properties.  As shown in \cite[Theorem 3.13]{fry2019tropical}, multicolored spaces are precisely the graphically stable moduli spaces for which the geometric tropicalization coming from the Pl\"ucker embedding agrees with the corresponding space of tropical curves (i.e. metric trees). Multicolored spaces also generalize the \emph{Losev--Manin} and \emph{heavy-light} spaces studied in \cite{losevmanin2000, weighted, KKL21}, which can be thought of as having one color for all the light points and a distinct color for each heavy point.  The heavy-light spaces themselves are precisely the Hassett spaces with the natural tropicalization properties above \cite{weighted}. 

Our first main result is a calculation of the alternating signed count of boundary strata for any graphically stable moduli space $\Moduli_{0, \Gamma}$. Topologically, this count gives the Euler characteristic of the boundary complex of $\Moduli_{0, \Gamma}$. For multicolored spaces, this is equivalent to the Euler characteristic of the link at the origin of the associated tropical moduli space.

We say a vertex of $\Gamma$ is \textbf{dominating} if it is connected to every other vertex. For $T \in \tree(\Gamma)$, we write $i(T)$ for the number of internal edges of $T$, i.e. the codimension of the corresponding stratum. We then have the following theorem.

\begin{thmg}
    Let $\Gamma$ be a simple graph with two dominating vertices $P, Q$. 
    Let $\ACO(\Gamma)$ be the set of acyclic orientations of $\Gamma \smallsetminus \{P, Q\}$. Then
    \[
    \sum_T (-1)^{i(T)} = (-1)^{n-1}|\ACO(\Gamma)|.
    \]
\end{thmg}

We prove Theorem \ref{thm:graphs} using two sign-reversing involutions. The first, which we call the {\bf up-down involution}, reduces the sum to a simpler signed count of certain colorings of $\Gamma$. This sum can be recognized as equalling $|\ACO(\Gamma)|$ by classic arguments such as contraction-deletion. Interestingly, a sign-reversing involution to yield the latter does not appear in the literature, although certain closely-related constructions of Gessel and Nadeau--Bernardi can in principle be iterated to construct one \cite{bernardi}. We accordingly exhibit a second involution that computes the latter equality directly, and is of combinatorial interest in its own right as it gives such an involution in closed form. This second involution, which we call the {\bf left-right involution}, depends on the total ordering of the vertices of $\Gamma$ and generalizes the well-known merge-split involution on ordered set partitions, which corresponds to the case where $\Gamma$ has no edges.

\begin{question}\label{question:simplicial}
    Determine how Theorem \ref{thm:graphs} generalizes to the simplicially-stable moduli spaces, the most general modular compactifications of $M_{0, n}$, which are governed by a choice of abstract simplicial complex $\mathcal{K}$ on the vertex set $[n]$ (see Definition \ref{defn:simplicial_stability}). 
\end{question}
We note that our first involution generalizes straightforwardly to these spaces (see Remark \ref{rmk:up-down-simplicial}). The alternating sum of strata can moreover be interpreted as an intersection product (see Remark \ref{rmk:all-ones-product} and Proposition \ref{prop:all-ones-product}).

We next consider intersection products of $\psi$ classes in the Chow ring of the moduli space:
\[\int_{\Mbar_{0, \Gamma}} \psi_1^{k_1} \cdots \psi_n^{k_n}, \text{ where } \sum k_i = n-3.\]
On $\Mbar_{0, n}$, such a product is always given by a multinomial coefficient. On $\Mbar_{0, \Gamma}$, a reduction formula to $\Mbar_{0, k}$ for various $k\leq n$ expresses these products as alternating sums involving certain set partitions (see \cite[Theorem 7.9]{alexeevguypsi} and Theorem \ref{thm:psi-formula-graph}). Many authors have studied other variants of this product, most commonly involving pullbacks of $\psi$ classes along forgetful maps, but staying on $\overline{M}_{0, n}$,  including \cite{CGM, silversmith2021crossratio, dastidar2022matroidpsiclasses, brakensiek, GGL-tournaments, GGL-hyperplanes, Reimer-Berg}. We discuss some similarities and contrasts between these results and ours further below.

We give a nonnegative evaluation of these integrals in the case of multicolored spaces $\Mbar_{0, [r_1, \ldots, r_m]}$. Our result again uses two successive sign-reversing involutions; as with Theorem \ref{thm:graphs}, the first is `easier', and the second is `harder' and again parallels the merge-split involution. We show:

\begin{thmm2}
    The intersection product \begin{equation*}
        \int_{\Moduli_{0,[r_1,\dots,r_m]}} \left(\prod_{i=1}^{\ell_1}\psi_i^{k_i^{(1)}}\right)\left(\prod_{i=r_1+1}^{r_1+\ell_2}\psi_i^{k_i^{(2)}}\right)\cdots \left(\prod_{i=r_1+\cdots+r_{m-1}+1}^{r_1+\cdots+r_{m-1}+\ell_m}\psi_i^{k_i^{(m)}}\right)   
    \end{equation*}
    equals the number of \textbf{fixed point decorations} with parameters $$\mathbf{k}=(k_1^{(1)},k_2^{(1)},\ldots,k_{n}^{(m)}).$$
    Here, we assume each $k_i^{(j)}$ is strictly positive, and so $\ell_j \geq 0$ is the number of distinct $\psi$ classes of color $C^{(j)}$ appearing in the product above.
\end{thmm2}

The fixed point decorations in Theorem \ref{thm:main2} may be defined as ways of displaying and
marking the numbers $1,\ldots,n$ as follows:

  \begin{itemize}
      \item The numbers are sorted into rows by color, and the first number in each of the first three rows is crossed out.  The numbers representing marked points $i$ with $k_i>0$ are boxed.
      \item We assign to each number that is not crossed out a label $C_i^{(j)}$, where $i$ is a choice of boxed number (of any color), which we call its {\bf hue}, subject to the following restrictions:
      
      \begin{itemize}
        \item \textbf{Multiplicities}: Each hue $C_i^{(j)}$ is assigned to $k_i^{(j)}$ numbers,
          \item \textbf{Mismatched}: Every unboxed number is assigned a hue not from its own color, and
          \item \textbf{No permission to merge}: There is no boxed number $r$, of color $j$ and hue $C_{t}^{(j)}$, such that $t<r$ and where $\{t\}$ has permission to merge as defined in Section \ref{sec:SRI2}.
      \end{itemize}
  \end{itemize}

In particular, an easy way to satisfy the third condition is if every boxed number $i$ can be assigned \emph{itself} as its hue $C_i^{(j)}$ while respecting the multiplicities and `mismatched' conditions on the unboxed numbers; see Example \ref{ex:special-fixed-points}. These particular decorations lead to our positivity results below.

\begin{example}\label{ex:fixed_point}
    Below is an example of a fixed point decoration for the product
    \begin{align*}
        \int_{\Moduli_{0,[6,5,4,2,2]}} \psi_1\psi_2\psi_3^2 \cdot \psi_7\psi_8^2\psi_9^3\psi_{10}^4\cdot \psi_{18}\psi_{19}.
    \end{align*}
 \begin{center}
    \begin{tabular}{cccccccccc}
    $C^{(1)}$ & = & $R$ & & $\boxed{\cancel{\phantom{.}{\color{red}1}\phantom{.}}}$ & $\boxed{\phantom{.}{\color{red}2}_{R_2}\phantom{.}}$ & $\boxed{{\phantom{.}{\color{red}3}_{R_{3}}\phantom{.}}}$ & ${\color{red}4}_{B_{9}}$ & ${\color{red}5}_{B_{10}}$ & ${\color{red}6}_{B_8}$ \\
    
    $C^{(2)}$ & = &$B$ & & $\boxed{\cancel{\phantom{.}\color{blue}7\phantom{.}}}$ & $\boxed{{\color{blue}8}_{B_{8}}}$ & $\boxed{{\color{blue}9}_{B_{9}}}$ & $\boxed{{\color{blue}10}_{B_{10}}}$ & ${\color{blue}11}_{R_3}$ & \\
    
$C^{(3)}$ & = & $G$ & &    $\cancel{{\color{mygreen}12}}$ & ${\color{mygreen}13}_{B_{10}}$ & ${\color{mygreen}14}_{R_1}$ & ${\color{mygreen}15}_{B_9}$ & & \\
    
   $C^{(4)}$ & = & $Y$ & & ${\color{y1}16}_{B_{10}}$ & ${\color{y1}17}_{B_7}$ & & & & \\
    
$C^{(5)}$ & = & $V$ & &  $\boxed{{\color{y2}18}_{V_{18}}}$ & $\boxed{{\color{y2}19}_{V_{19}}}$ & & & &
\end{tabular}
\end{center}
In this decoration, each boxed number has been assigned itself as its hue, and all unboxed numbers have mismatched hues. It is therefore a fixed point decoration.
\end{example}

Our description specializes well to several special cases of interest, as outlined in Section \ref{sec:examples}. We note that these enumerations are interesting even for heavy-light spaces (see Section \ref{subsec:single-color-heavy-light}).

Theorem \ref{thm:main2} implies in particular that products of psi classes on multicolored spaces are always nonnegative. (We remark that this nonnegativity does not hold on all modular compactifications of $M_{0, n}$.) As a further application of Theorem \ref{thm:main2}, we give necessary and sufficient conditions for when such a product is positive. These conditions follow easily from Theorem \ref{thm:main2} and turn out to be linear inequalities in the data $\mathbf{k}$, convexity conditions similar to the \emph{Cerberus condition} \cite{brakensiek} for products of pullbacks of $\psi$ classes on $\overline{M}_{0, n}$, the \emph{Catalan condition} \cite{CGM} for products of omega classes, and convex quantities arising in the calculation of related integrals in terms of mixed volumes \cite{dastidar2022matroidpsiclasses}.

\begin{thmnz}
    Let $k_{C^{(j)}} = \sum k_i^{(j)}$ be the sum of the exponents on the $\psi$ classes of color $C^{(j)}$, with $k_{C^{(1)}} + \cdots + k_{C^{(m)}} = n-3$. 

    The intersection product in Theorem \ref{thm:main2} is nonzero if and only if the inequalities
    $$k_{C^{(j)}}\le \ell_j+n -3 - r_j$$ hold for all $j$ such that $k_{C^{(j)}} \ne 0$.
\end{thmnz}

\begin{example} \label{ex:inequalities}
The $\psi$ class intersection $\int \psi_1\psi_2\psi_3^2 \cdot \psi_7\psi_8^2\psi_9^3\psi_{10}^4\cdot \psi_{18}\psi_{19}$ in Example \ref{ex:fixed_point} had 
\begin{align*}
\ell_1=3, \quad \ell_2=4, \quad &\ell_3=\ell_4=0, \quad \ell_5=2, \\
k_{C^{(1)}}=4, \quad k_{C^{(2)}}=10, \quad &k_{C^{(3)}}=k_{C^{(4)}}=0, \quad k_{C^{(5)}}=2.
\end{align*}
The inequalities of Theorem \ref{thm:nonzero} say
\begin{alignat*}{2}
    C^{(1)}&: \quad & 4 &\le 3+19-3-6=13, \\
    C^{(2)}&: \quad & 10 &\le 4+19-3-5=15, \\
    C^{(5)}&: \quad & 2 &\le 2+19-3-2 = 16,
\end{alignat*}
with no conditions on $C^{(3)}$ and $C^{(4)}$ since $k_{C^{(3)}} = k_{C^{(4)}} = 0$. By Theorem \ref{thm:nonzero}, we conclude that the integral is nonzero.
\end{example}

Our methods yield both an upper bound and a lower bound for the integral:

\begin{cor2}
    The intersection product of Theorem \ref{thm:main2}
    is bounded below by the number of {\bf mismatched colorings} and bounded above by the number of {\bf mismatched decorations with singleton hues}, with parameters $\mathbf{k}$.
\end{cor2}

Both of the terms above are defined in Section \ref{sec:cerberus}. Our proof of Theorem \ref{thm:nonzero} makes use of Hall's Marriage Lemma for bipartite matchings, which likewise parallels \cite{brakensiek} on Kapranov degrees and \cite{silversmith2021crossratio} on cross-ratio degrees on the standard moduli space $\overline{M}_{0, n}$, both of which give upper bounds in terms of bipartite matchings (whereas we give both upper and lower bounds). See Remark \ref{rmk:comparison-to-BELL-SIL} for further discussion.

\subsection{Outline}

After establishing background and notation on graphically stable moduli spaces in Section \ref{sec:background-graphical}, $\psi$ classes in Section \ref{sec:background-psi}, and the combinatorics of sign-reversing involutions in Section \ref{sec:background-SRI}, we prove Theorem \ref{thm:graphs} in Section \ref{sec:SRI1}. We then prove Theorem \ref{thm:main2} in Section \ref{sec:SRI2} and give some applications and observations in Section \ref{sec:examples}.  Finally, we prove Theorem \ref{thm:nonzero} in Section \ref{sec:cerberus}.

\subsection{Acknowledgments}
 This work arose from the Banff workshop on Combinatorics of Moduli of Curves (COMOC) in summer 2024.  We thank Renzo Cavalieri, Ian Cavey, Deniz Genlik, Sean Griffin, Matt Larson, Diane Maclagan, Rohini Ramadas, and Rob Silversmith for helpful conversations pertaining to this work.

\section{Geometric and combinatorial backgrouhnd} 

\subsection{Background on graphical moduli spaces}\label{sec:background-graphical}

In order to discuss the graphically stable spaces $\Mbar_{0,\Gamma}$, we first recall some basic facts about moduli spaces of curves and some of their compactifications. Interested readers should consider consulting any of \cite{hm1998, vakil08, acgh2013} for a more thorough introduction to these spaces.

\begin{defn}\label{defn:deligne-mumford_stability}
    Denote by $M_{0,n}$ the moduli space of smooth rational curves with $n$ distinct, ordered marked points up to isomorphism. We say that an at-worst nodal curve of arithmetic genus zero with $n$ distinct, smooth marked points is \textbf{Deligne-Mumford stable} if its automorphism group is finite. Equivalently, such a curve consists of a tree structure of $\PP^1$s glued at nodes with marked points distributed throughout; the curve is Deligne-Mumford stable if each $\PP^1$ has at least three ``special points'' (nodes or marked points). Then $\Mbar_{0,n}$ is the compactification of $M_{0,n}$ consisting of all such Deligne-Mumford stable curves.  
\end{defn}

The Deligne-Mumford compactification $\Mbar_{0,n}$ is a smooth, connected variety of dimension $n-3$, which can be realized as an iterated blow-up of projective space \cite{Ka2}. In fact, it is a wonderful compactification of the braid arrangement of hyperplanes in $\PP^{n-3}$ \cite{DeConciniProcesi}.

Several families of alternative compactifications for $M_{0,n}$ have been discussed in the literature \cite{losevmanin2000, hassett2003, smyth_zstable, BlankersBozlee} by disallowing certain dual trees and allowing some of the marked points to collide. Of particular interest in this paper are \textbf{graphically stable spaces} introduced in \cite{fry2019tropical} and some of their specializations. We also discuss two related compactifications: \textbf{Hassett spaces} introduced in \cite{hassett2003} and \textbf{simplicially stable spaces} introduced in \cite{BlankersBozlee}.

In each case, some marked points on curves parametrized by these spaces are allowed to collide: in a graphically stable space, we start with a graph $\Gamma$ on vertices labeled by $[n]$ and allow marks to collide if their corresponding indices form an independent set in $\Gamma$; in a Hassett space, we choose a vector of rational weights $\mathcal{A} = (a_1,\dots,a_n)$ and allow marks to collide if their corresponding weights are small enough; and in a simplicially stable space, we choose a simplicial complex $\mathcal{K}$ on the ground set $[n]$ and allow marks to collide if their corresponding indices are a face of $\mathcal{K}$. Though the results of this paper concern graphically stable spaces and some specializations thereof, we rely on a generalization of a result on Hassett spaces in Section \ref{sec:background-psi}, so we include all there definitions here for completeness.

\begin{defn}[{\cite{fry2019tropical}}]\label{defn:graphical_stability}
    Let $\Gamma$ be a simple graph on vertices $[n]$ that is not bipartite. We say that $(C;p_1,\dots,p_n)$, an at-worst nodal curve of arithmetic genus zero with $n$ smooth marked points, is \textbf{$\Gamma$-stable} if
    \begin{itemize}
        \item for every irreducible $Z \subseteq C$ with two nodes, there exists at least one marked point on $Z$; 
        \item for every irreducible $Z\subseteq C$ with a single node, there exist at least two marked points $p_i$ and $p_j$ on $Z$ such that $i$ and $j$ have an edge in $\Gamma$; and
        \item for each smooth $x\in C$, the set $\{i \, : \, p_i = x\}$ forms an independent set in $\Gamma$.
    \end{itemize}
    The \textbf{graphically stable space} $\Mbar_{0,\Gamma}$ is then the compactification of $M_{0,n}$ consisting of all $\Gamma$-stable curves up to isomorphism.
\end{defn}

\begin{defn}[{\cite{hassett2003}}]\label{defn:hassett_stability}
    Fix \textbf{weight data} $\mathcal{A} = (a_1,\dots,a_n)$ with $a_i\in (0,1]$ with $\sum a_i > 2$. We say that $(C;p_1,\dots,p_n)$, an at-worst nodal curve of arithmetic genus zero with $n$ smooth marked points, is \textbf{$\mathcal{A}$-stable} if
    \begin{itemize}
        \item for every irreducible $Z \subseteq C$ with two nodes, there exists at least one marked point on $Z$; 
        \item for every irreducible $Z\subseteq C$ with a single node, we have $\sum_{p_i \in Z} a_i > 1$; and
        \item for each smooth $x\in C$, we have $\sum_{p_i x} a_i \leq 1$.
    \end{itemize}
    Then the \textbf{Hassett space} $\Mbar_{0,\mathcal{A}}$ is the compactification of $M_{0,n}$ consisting of all such $\mathcal{A}$-stable curves up to isomorphism.
\end{defn}

\begin{defn}[{\cite{BlankersBozlee}}]\label{defn:simplicial_stability}
    Fix a simplicial complex $\mathcal{K}$ with $0$-skeleton $[n]$ such that $[n]$ is not the union of two faces of $\mathcal{K}$. We say that $(C;p_1,\dots,p_n)$, an at-worst nodal curve of arithmetic genus zero $n$ smooth marked points is \textbf{$\mathcal{K}$-stable} if
    \begin{itemize}
        \item for every irreducible $Z\subseteq C$ with two nodes, there exists at least one marked point on $Z$;
        \item for every irreducible $Z\subseteq C$ with a single node, there exist marked points $p_i, p_j \in Z$ such that $\{i, j\}$ is not an edge of $\mathcal{K}$; and
        \item for each smooth $x\in C$, the set $\{i \, : \, p_i = x\}$ is a face of $\mathcal{K}$.
    \end{itemize}
    The \textbf{simplicially stable space $\Moduli_{0,\mathcal{K}}$} is then the compactification of $\moduli_{0,n}$ consisting of all $\mathcal{K}$-stable curves up to isomorphism.
\end{defn}

\begin{remark}\label{rem:simplicial_stability}
    There are Hassett spaces which cannot be realized as graphically stable spaces and vice versa. The more general notion of simplicial stability encompasses both Hassett and graphical stability. However, there do exist compactifications that can be expressed via any of these three. Notably, the Deligne-Mumford compactification $\Moduli_{0,n}$, corresponds to any of
    \begin{itemize}
        \item $\Gamma$ a complete graph,
        \item $\mathcal{A} = (1,1,\dots,1)$, or
        \item $\mathcal{K}$ a simplicial complex with only a $0$-skeleton;
    \end{itemize}
    and the \textbf{Losev-Manin space} $\mathrm{LM}_{2|n-2}$, originally introduced in \cite{losevmanin2000}, which corresponds to
    \begin{itemize}
        \item $\Gamma$ with two vertices with maximum valence and no other edges,
        \item $\mathcal{A} = (1,1, \frac{1}{n-2}, \dots \frac{1}{n-2})$, or
        \item $\mathcal{K}$ consisting of an $(n-2)$-simplex and two disjoint vertices.
    \end{itemize}
\end{remark}

For each of these stabilities, we can define some important subspaces by specifying some combinatorial data in the form of dual graphs -- since we are working in genus zero, we in fact only need to use dual trees. Note that these trees are  distinct from the graph $\Gamma$ used to define graphical stability.

\begin{defn}\label{defn:dual_tree}
    Let $[C;p_1,\dots,p_n]$ be a point in $\Moduli_{0,\mathcal{K}}$ with $p_i \neq p_j$ for $i\neq j$. The \textbf{dual tree} of $[C;p_1,\dots,p_n]$ is the tree with one vertex for every $\PP^1$ component of $C$, an edge connecting vertices when corresponding components share a node, and labeled half-edges (leaves) corresponding to the marked points $p_1,\dots,p_n$. We say that a vertex with a half edge labeled $p_i$ \textbf{contains} $p_i$. A \textbf{boundary stratum} of $\Mbar_{0,\mathcal{K}}$ is a subset of curves sharing the same dual tree.
    \end{defn}
When working with a graphically stable space $\Moduli_{0,\Gamma}$, the set $\tree(\Gamma)$ of \textbf{$\Gamma$-stable trees} consists of all possible dual trees for a given $\Gamma$.

\begin{example}\label{ex:graph-diagram}
    Below is a graph $\Gamma$, a stable curve in $\Mbar_{0,\Gamma}$, and its dual tree.
\begin{center}
    \includegraphics{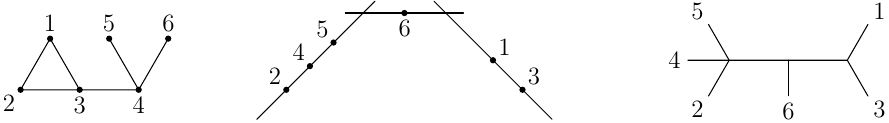}
\end{center}
Note that $\{1, 3\}$ and $\{2, 4, 5\}$ are not independent sets of $\Gamma$.
\end{example}

In \cite{fry2019tropical}, Fry shows that specific graphically stable spaces behave particularly well in regard to tropicalization. It is for these \textbf{multicolored spaces} that our results in Section \ref{sec:SRI2} apply.

\begin{defn}\label{defn:multicolored}
    Let $\Gamma$ be a complete multipartite graph on $[n]$ with independent sets $C^{(1)}, \dots, C^{(m)}$ of respective sizes $r_1,\dots,r_m$ so that $r_1 + \cdots + r_m = n$. We call $\Mbar_{0,[r_1,\dots,r_m]} := \Mbar_{0,\Gamma}$ a \textbf{multicolored space}. We interpret each marked point $p_i$ on a curve (or corresponding leaf of the dual tree) as having color $C^{(j)}$ if $i\in C^{(j)}$.
\end{defn}

\begin{remark} \label{rmk:dual-trees}
    Dual trees of $\Mbar_{0,[a_1,\ldots,a_k]}$ have the property that whenever an edge is removed from the tree, both remaining connected components contain a pair of leaves of different colors.
\end{remark}

\subsection{Background on \texorpdfstring{$\psi$}{} classes}\label{sec:background-psi}

Throughout this section we refer frequently to \cite{alexeevguypsi}, which nominally discusses $\psi$ classes on (stable maps from) Hassett spaces. However, for much of their paper the authors use the language of simplicially stable spaces -- though these were only formally defined and constructed in \cite{BlankersBozlee}. As such, many of the arguments and results in \cite{alexeevguypsi} immediately generalize to simplicially stable spaces. Below, we summarize and translate the relevant results to our setting and notation; an interested reader may consult their original article, along with \cite{bc_hassett} for another treatment of some of their intersection results and \cite{BlankersBozlee} for the formal construction of simplicially stable spaces.

\begin{remark}\label{rem:higher-genus}
    Other than Lemma \ref{lem:multinomial-form}, the contents of this section also generalize to moduli spaces of higher genus curves $\moduli_{g,n}$ without alteration. By restricting our attention to genus zero, we benefit from the multinomial coefficient description of $\psi$ class intersections on $\Moduli_{0,n}$ used in Lemma \ref{lem:multinomial-form}, and we need only consider dual trees instead of more general dual graphs when studying boundary strata in Section \ref{sec:SRI1}.
\end{remark}

\begin{defn}\label{defn:psi_classes}
    Let $\Moduli$ be any simplicially stable compactification of $\moduli_{0,n}$, and let $\mathbb{L}_i \to \Moduli$ be the \textbf{$i$th cotangent line bundle}, whose fibers are naturally identified over a point $[C;p_1,\dots,p_n]\in \Moduli$ with the cotangent space $T^*_{p_i}C$. Define the \textbf{$i$th $\psi$ class} in $A^1(\Moduli)$ as
    \begin{align*}
        \psi_i := c_1(\mathbb{L}_i),
    \end{align*}
    where $c_1$ is the first Chern class.
\end{defn}

While originally studied on the Deligne-Mumford compactification, $\psi$ classes can be defined in this way for any simplicially stable space, since the marked points
are required to be smooth. In \cite[Theorem 7.9]{alexeevguypsi}, the authors give a formula for top degree $\psi$ class intersections on Hassett spaces in terms of $\psi$ class intersections on $\Moduli_{g,n}$ (see also \cite[Corollary 3.8]{bc_hassett}, from which we adopt the bulk of our notation). 
\begin{thm}[{{\cite[Theorem 7.9]{alexeevguypsi}, \cite[Corollary 3.8]{bc_hassett}}}]\label{thm:psi-formula-og}
    For weight data $\mathcal{A} = (a_1,\dots,a_n)$, let $\mathfrak{P}_{\mathcal{A}}$ be the set of partitions $\mathcal{P} = \{P_1,\dots,P_r\} \vdash [n]$ such that $\sum_{i\in P_j} a_i \leq 1$ for each $P_j$.  We write $\ell(\mathcal{P})$ for the number of blocks $r$ of the set partition.
    For $i \in [n]$, let $k_i$ be a non-negative integer, and let $\sum_{i\in \cup [n]} k_i = n-3$. Define $k_{P_j} = \sum_{i\in P_j} k_i$. Then
    \begin{align*}
        \int_{\Moduli_{0,\mathcal{A}}} \prod_{i=1}^n \psi_i^{k_i} = \sum_{\mathcal{P} \in \mathfrak{P}_{\mathcal{A}}} (-1)^{n + \ell(\mathcal{P})} \int_{\Moduli_{0,\ell(\mathcal{P})}} \prod_{j=1}^{\ell(\mathcal{P})} \psi_{j}^{k_{P_j} - |P_j| + 1}.
    \end{align*}
\end{thm}

The arguments in \cite{alexeevguypsi} involving $\psi$ classes proceed essentially unchanged in the case of simplicially stable spaces -- even those which are not Hassett spaces -- with the caveat that the conclusions in terms of Hassett stability must be reworded to be given in terms of simplicial stability. This result in turn can be restated in terms of graphically stable spaces, per Remark \ref{rem:simplicial_stability}. In particular, we need only change the definition of the set of partitions $\mathfrak{P}$ in the sum.

\begin{thm}\label{thm:psi-formula-graph}
    For a simple graph $\Gamma$ on vertex set $[n]$, let $\mathfrak{P}_{\Gamma}$ be the set of partitions $\mathcal{P} = \{P_1,\dots,P_r\} \vdash [n]$ 
    such that $P_j$ is an independent set in $\Gamma$ for each $j$.
    Then
    \begin{align*}
        \int_{\Moduli_{0,\Gamma}} \prod_{i=1}^n \psi_i^{k_i} = \sum_{\mathcal{P} \in \mathfrak{P}_{\Gamma}} (-1)^{n + \ell(\mathcal{P})} \int_{\Moduli_{0,\ell(\mathcal{P})}} \prod_{j=1}^{\ell(\mathcal{P}))} \psi_{j}^{k_{P_j} - |P_j| + 1}.
    \end{align*}
\end{thm}
\begin{proof}
    Although their goal is to address Hassett spaces, the arguments in \cite[Corollary 4.9]{alexeevguypsi} and all required antecedent propositions are written in terms of simplicial stability; therefore, for any graphically stable space (which is in particular also a simplicially stable space), there exists a map $\rho:\Moduli_{0,n} \to \Moduli_{0,\Gamma}$ which consists of a series of blow-ups. The same observation holds for \cite[Corollary 5.4]{alexeevguypsi}, establishing that
    \begin{align*}
        \psi_i = \rho^*(\psi_i) + D_i,
    \end{align*}
    where the correction term $D_i$ is the sum of all boundary divisors corresponding to dual trees $T$ for which $\{j \in [n] : j$ is attached to the same vertex of $T$ as $i\}$ forms an independent set of $\Gamma$.

    With this pullback statement, the results \cite[Lemma 2.8, Lemma 2.9]{bc_hassett} comparing the intersection of $\psi$ classes with boundary strata also follow through analogously. Finally, the argument in \cite[Corollary 3.8]{bc_hassett} requires only these results and some combinatorial bookkeeping independent of any geometry, so its analogue -- which is the theorem statement -- also follows.  
\end{proof}

Changing the underlying combinatorial structure (simplicial complex, weights, graph) defining stability does not change the value of an integral of $\psi$ classes if the changes in the structure occur away from the indices corresponding to $\psi$ classes with positive exponents. This follows from the fact that the morphism $\rho:\Moduli \to \Moduli'$ induced also induces a map of universal families $\tilde{\rho}:\overline{U} \to \overline{U}'$, which is an isomorphism in a neighborhood of the $i$th section if the $i$th vertex/weight is not impacted by the change in stability conditions. Since we are primarily interested in $\psi$ class integrals on graphical spaces, we present the result in that context, but the respective statements and proofs for the Hassett and simplicially stable cases are entirely analogous.

\begin{cor}\label{cor:color-clumping}
   Let $\Gamma$ be a simple graph on vertex set $[n]$, and let $\Gamma'$ be a graph obtained from $\Gamma$ by deleting some set of edges $E$ of $\Gamma$. Let $\mathbf{k} = (k_1, \ldots, k_n)$ be nonnegative integers. Suppose that for all $i \in [n]$ with $k_i > 0$, $i$ is not incident to any edge in $E$. Then
    \begin{align*}
        \int_{\Moduli_{0,\Gamma}} \prod_{i=1}^n \psi_i^{k_i} = \int_{\Moduli_{0,\Gamma'}} \prod_{i=1}^n \psi_i^{k_i}.
    \end{align*}
\end{cor}
\begin{proof}

    Let $\rho:\Moduli_{0,\Gamma} \to \Moduli_{0,\Gamma'}$ be the reduction morphism which, on the level of curves, collapses subcurves which are $\Gamma$-stable but not $\Gamma'$ stable. Similar to the previous proof, we have 
    \begin{align*}
        \psi_i = \rho^*(\psi_i) + D_i,
    \end{align*}
    where $D_i$ is the sum of all boundary divisors with dual trees $T$ for which $\{j \in [n] : j$ is attached to the same vertex of $T$ as $i\}$ forms an independent set in $\Gamma'$ but not in $\Gamma$. If $k_i > 0$, then $i \in \Gamma$ is not incident to any of the deleted edges, so the independent sets containing $i$ are the same in $\Gamma$ as in $\Gamma'$.
    Thus $D_i = 0$ and so all terms appearing in the integral have $\psi_i = \rho^*(\psi_i)$.
\end{proof}

We now fix $\Gamma$ to be a complete multipartite graph so that $\Moduli_{0,\Gamma}$ a multicolored space.

\begin{lemma}\label{lem:multinomial-form}
Let $\Mbar_{0,[r_1,\dots,r_m]}$ be a multicolored space with $r_1 + \cdots + r_m = n$. We have
$$\int_{\Mbar_{0,[r_1,\ldots,r_m]}}\psi_1^{k_1}\psi_2^{k_2}\cdots \psi_n^{k_n}=\sum_{\mathcal{P}}(-1)^{\sum_{B\in \mathcal{P}}(|B|-1)}\binom{\ell(\mathcal{P})-3}{k_{B_1}-|B_1|+1,k_{B_2}-|B_2|+1,\ldots}$$    
where the sum is over all set partitions $\mathcal{P}$ of the marked points such that each block $B$ is monochromatic. Here we write $k_B=\sum_{i\in B}k_i$, and when some $k_{B_j}-|B_j|+1$ is negative, the entire multinomial coefficient is defined to be $0$.
\end{lemma}

\begin{proof}
    We use Theorem \ref{thm:psi-formula-graph} for the complete multipartite graph with independent sets of sizes $r_1,\ldots,r_m$. In genus $0$, it is well known (e.g., \cite{k:psi}) that the $\psi$ class intersection on the right hand side is simply the multinomial coefficient shown.
\end{proof}

\begin{remark}\label{rmk:coefficients-nonzero}
    The terms in the summation in Lemma \ref{lem:multinomial-form} are $0$ whenever $k_{B_j}-|B_j|+1<0$ for some $j$.  Thus we may restrict our summation to the set partitions for which each block $B_j$ has the property that $k_{B_j}-|B_j|+1\ge 0$. In particular, this means that any block that does not contain a point with a $\psi$ class occurring to a positive power in the product must be a singleton block.
\end{remark}

\subsection{Background on sign reversing involutions (SRIs)}\label{sec:background-SRI}

We now introduce the combinatorial notation we use throughout.  See \cite[Section 2.2]{SaganBook} for many of these definitions and for the ordered set partition involution below.

\begin{defn}
    Let $X$ be a set and \[\epsilon : X \to \{\pm 1\}\] a sign statistic on $X$. We recall that a {\bf sign-reversing involution} is an involution $\varphi : X \to X$ such that whenever $\varphi(x) \ne x$, then $\epsilon(\varphi(x)) = -\epsilon(x)$.
\end{defn} 

The main purpose of constructing a sign-reversing involution is to achieve the simplification
\begin{equation}\label{eq:alt-sum-X}
\sum_{x \in X} \epsilon(x) = \sum_{x \in \Fix(\varphi)} \epsilon(x),
\end{equation}
where $\Fix(\varphi) = \{x : \varphi(x) = x\}$ is the set of fixed points. In most cases of interest, $\varphi$ is designed such that its fixed points all have the same sign, in which case
\[
\sum_{x \in \Fix(\varphi)} \epsilon(x) = \pm |\Fix(\varphi)|.
\]

\subsection{(Ordered) set partitions}

Our results generalize a well-known sign reversing involution on ordered set partitions. We establish notation for set partitions and ordered set partitions here and recall the involution below.

\begin{defn}
    A \textbf{set partition} of a set $X$ is a set $\mathcal{P}=\{B_1,\ldots,B_k\}$ of subsets $B_i\subseteq X$ such that:
    \begin{itemize}
        \item $|B_i|>0$ for all $i$,
        \item $B_i\cap B_j=\emptyset$ for any $i,j$, and
        \item $B_1\cup \cdots \cup B_k=X$.
    \end{itemize}
    The $B_i$ are called \textbf{blocks} of the set partition, and $k=|\mathcal{P}|$ is also denoted $\ell(\mathcal{P})$ for the \textbf{length} of the set partition, or the number of blocks.
\end{defn}

\begin{defn}
    An \textbf{ordered set partition}, or OSP, is a set partition along with an ordering on the blocks, forming a tuple $P=(B_1,\ldots,B_k)$ with the properties above.   Its \textbf{sign} is $$(-1)^{n-\ell(P)}=(-1)^{\sum |B_i|-1}.$$  Note that this means when $P$ consists of all singleton blocks, its sign is $1$.
\end{defn}

For example, the set $\{\{4\},\{2\},\{1,3,6\},\{5\},\{7,8\}\}$ is a set partition of $\{1,2,\ldots,8\}$, whereas $(\{4\},\{2\},\{1,3,6\},\{5\},\{7,8\})$ is an ordered set partition, with sign $(-1)^{3}=-1$.

\begin{defn} \label{def:invOSP}
    The sign reversing involution $\invOSP$ is defined on ordered set partitions as follows.  Given an OSP $P$, let $B_i$ be the leftmost block that is either 
    
    \begin{itemize}
        \item not a singleton ($|B_i|>1$) or 
        \item a singleton $B_i=\{c\}$ that occurs just after a smaller singleton $\{m\}$.
    \end{itemize}
       In either case let $B_{i-1}=\{m\}$ be the singleton block preceding $B_i$ (if it exists) and let $c=\min(B_i)$.

    \textbf{Split case:} If $c<m$ or $m$ does not exist, replace $B_i$ by the two blocks $\{c\},B_i\setminus c$ in order.

    \textbf{Merge case:} If $c>m$, merge $m$ into block $B_i$.

    \noindent We define the output to be $\invOSP(P)$.
\end{defn}

\begin{example}
    Starting from $P=(\{4\},\{2\},\{1,3,6\},\{5\},\{7,8\})$, we identify the leftmost block $B_i$ to be $\{1,3,6\}$, since the singleton $2$ is less than the preceding singleton $4$.  So we have $$\invOSP(P)=\{\{4\},\{2\},\{1\},\{3,6\},\{5\},\{7,8\}\}.$$
    To apply the involution again, we now have that $\{3,6\}$ is the leftmost block that is either not a singleton or a singleton whose previous singleton is smaller.  But now we are in the merge case because $1<3$.  So we have 
    $$\invOSP\left(\invOSP(P)\right)=\{\{4\},\{2\},\{1,3,6\},\{5\},\{7,8\}\}=P.$$
    Note that the sign also changes when we merge or split.
\end{example}

Note that the involution $\invOSP$ allows one to cancel terms in the alternating sum below to find: $$\sum_{P\in \mathrm{OSP}(n)} (-1)^{n-\ell(P)}=1.$$  Indeed, there is only one fixed point of $\invOSP$, namely when there are no blocks $B_i$ satisfying either condition.  Thus the blocks must all be singletons and in decreasing order, which only happens in the unique OSP $(\{n\},\ldots,\{2\},\{1\})$.

\section{An SRI for alternating counts of boundary strata}\label{sec:SRI1}

Let $\Gamma$ be a simple graph on vertices $\{P,Q\} \cup [n]$ in which $P$ and $Q$ are each dominant - they are connected to every other vertex in $\Gamma$.
\begin{thm}\label{thm:graphs}
    We have $$\sum_{T\in \tree(\Gamma)}(-1)^{i(T)}=(-1)^{n-1}|\ACO(\Gamma)|$$
    where $i(T)$ is the number of internal edges of $T$ and $\ACO(\Gamma)$ is the set of acyclic orientations of $\Gamma\smallsetminus \{P,Q\}$.
\end{thm}

To set up the proof, we first recall (Remark \ref{rmk:dual-trees}) that the dual trees we are summing over are precisely those such that for every internal edge $e$, each of the branches $B,B'$ on either side of $e$ contains a pair of leaves whose corresponding vertices in $\Gamma$ are joined by an edge.

Since in any tree there is a unique path between any pair of vertices, we may draw our dual trees as a path from $P$ to $Q$, with branches coming off vertices along this path.  The above condition means that any branch off the path contains some pair of leaves corresponding to a pair of vertices connected by an edge in $\Gamma$. In particular, every edge along the path from the vertex containing $P$ to the vertex containing $Q$ has this property satisfied, since $P$ and $Q$ are dominant in $\Gamma$. An example is shown below; on the left is the dual tree, and on the right is the graph $\Gamma \smallsetminus \{P,Q\}$.

\begin{center}
\begin{asy}
import math;
import graph;
unitsize(1cm);

draw((0,0)--(4,0));

label("$P$",(0,0),W);
label("$Q$",(4,0),E);

draw((1,0)--(1,-1)--(0.5,-1.5));
draw((1,-1)--(1.5,-1.5)--(2,-2));
draw((1.5,-1.5)--(1,-2));
draw((2,0)--(2,-0.5));
draw((2.7,-0.5)--(3,0)--(3.3,-0.5));

label("$2$",(0.5,-1.5),S);
label("$4$",(1,-2),S);
label("$6$",(2,-2),S);
label("$3$",(2,-0.5),S);
label("$1$",(2.7,-0.5),S);
label("$5$",(3.3,-0.5),S);

draw((6,0)--(7,0)--(6.5,-0.8)--cycle);
draw((8,0)--(8.5,-0.5));

dot((6,0));
dot((7,0));
dot((6.5,-0.8));
dot((8,0));
dot((8.5,-0.5));
dot((9,0));

label("$1$",(6,0),W);
label("$2$",(7,0),E);
label("$4$",(6.5,-0.8),S);
label("$3$",(8,0),N);
label("$5$",(8.5,-0.5),SE);
label("$6$",(9,0),E);
\end{asy}    
\end{center}

We now construct a sign reversing involution $\invheavy$ on this set of trees, which we call $\Trees(\Gamma)$, in two steps.

\begin{defn}\label{def:up-down}  The \textbf{up-down involution} $\invheavy$ is defined on $\tree(\Gamma)$ as follows. Let $v_1, \ldots, v_r$ be the vertices along the path $p$ from $P$ to $Q$. For each $i$, let $L_i$ be the set of leaves attached to the branches at $v_i$ off the path from $P$ to $Q$.  If any $L_i$ has a pair of leaves connected by an edge in $\Gamma$, we choose the smallest such $i$ and apply the following process to define $\invheavy(T)$:
\begin{itemize}
    \item \textbf{Contraction:} If there is only one edge attached to $v_i$ besides the two along the path $p$, contract that edge.
    \item \textbf{Expansion:} If there are at least two edges attached to $v_i$ besides the two along the path $p$, let $u_1,\ldots,u_k$ be the vertices connected to $v_i$ not along the path $p$.  Disconnect all the $u_j$'s from $v_i$, connect them all to a new vertex $v'$, and connect $v'$ to $v_i$.
\end{itemize}

If no such $L_i$ exists, we define $\invheavy(T)=T$.

\begin{center}
\begin{asy}
import math;
import graph;
unitsize(1cm);

draw((0,0)--(4,0));

dot((1,0));

label("$v_i$",(1,0),N);

label("$P$",(0,0),W);
label("$Q$",(4,0),E);

draw((1,0)--(1,-1)--(0.5,-1.5));
draw((1,-1)--(1.5,-1.5)--(1.8,-2));
draw((1.5,-1.5)--(1,-2));
draw((2,0)--(2,-0.5));
draw((2.7,-0.5)--(3,0)--(3.3,-0.5));

label("$2$",(0.5,-1.5),S);
label("$4$",(1,-2),S);
label("$6$",(1.8,-2),S);
label("$3$",(2,-0.5),S);
label("$1$",(2.7,-0.5),S);
label("$5$",(3.3,-0.5),S);

draw((5,-0.25)--(6,-0.25),arrow=Arrows(TeXHead));

currentpicture=shift((-7,0))*currentpicture;

draw((0,0)--(4,0));

dot((1,0));

label("$v_i$",(1,0),N);

label("$P$",(0,0),W);
label("$Q$",(4,0),E);

draw((1,0)--(0.5,-0.5));
draw((1,0)--(1.5,-0.5));
draw((1,-1)--(1.5,-0.5)--(1.8,-1));
draw((2,0)--(2,-0.5));
draw((2.7,-0.5)--(3,0)--(3.3,-0.5));

label("$2$",(0.5,-0.5),S);
label("$4$",(1,-1),S);
label("$6$",(1.8,-1),S);
label("$3$",(2,-0.5),S);
label("$1$",(2.7,-0.5),S);
label("$5$",(3.3,-0.5),S);

\end{asy}    
\end{center}
\end{defn}

Then $\invheavy$ is a sign-reversing involution that cancels all terms in the summation other than the trees defined below.

\begin{defn}
    Let $T \in \tree(\Gamma)$ and let $v_1, \ldots, v_r$ be as above. We say $T$ is a \textbf{$\Gamma$-millipede} if each $v_i$ is attached only to leaves, and these leaves $L_i$ form an independent set in $\Gamma$. We write $\millipedes(\Gamma)$ for the set of $\Gamma$-millipedes.
\end{defn}
We note that a $\Gamma$-millipede is equivalent to the data of an ordered set partition of the vertices of $\Gamma \smallsetminus \{P, Q\}$, for which each block is an independent set.  An example of a $\Gamma$-millipede is shown below, for $\Gamma\smallsetminus \{P,Q\}$ the graph on the right.

\begin{center}
\begin{asy}
import math;
import graph;
unitsize(1cm);

draw((0,0)--(4,0));

label("$P$",(0,0),W);
label("$Q$",(4,0),E);

draw((0.5,-0.5)--(1,0)--(1.5,-0.5));
draw((1,0)--(1,-0.5));
draw((2,0)--(2,-0.5));
draw((2.7,-0.5)--(3,0)--(3.3,-0.5));

label("$2$",(0.5,-0.5),S);
label("$3$",(1.5,-0.5),S);
label("$4$",(2,-0.5),S);
label("$1$",(2.7,-0.5),S);
label("$5$",(3.3,-0.5),S);
label("$6$",(1,-0.5),S);

draw((6,0)--(7,0)--(6.5,-0.8)--cycle);
draw((8,0)--(8.5,-0.5));

dot((6,0));
dot((7,0));
dot((6.5,-0.8));
dot((8,0));
dot((8.5,-0.5));
dot((9,0));

label("$1$",(6,0),W);
label("$2$",(7,0),E);
label("$4$",(6.5,-0.8),S);
label("$3$",(8,0),N);
label("$5$",(8.5,-0.5),SE);
label("$6$",(9,0),E);

\end{asy} 
\end{center}

\begin{remark}[Up-down involutions on simplicially stable spaces]\label{rmk:up-down-simplicial}
A similar up-down involution makes sense on any simplicially-stable space $\Mbar_{0, \mathcal{K}}$ with two isolated vertices $P, Q$ (i.e. marked points that cannot collide with any other point). We define $v_1, \ldots, v_r$ and $L_1, \ldots, L_r$ as in Definition \ref{def:up-down}. If any $L_i$ is a non-face of $\mathcal{K}$, we take the smallest such $i$ and apply the two cases above. The fixed points correspond to ordered set partitions of $\mathcal{K} \smallsetminus \{P, Q\}$ for which every block is a face of $\mathcal{K}$.
\end{remark}

We now define a sign-reversing involution on $\millipedes(\Gamma)$.

\begin{defn}\label{def:disconnected-region}
Let $T \in \millipedes(\Gamma)$, let $v_1, \ldots, v_r$ be as above, and let $c$ be a leaf attached at some $v_i$. Let $j$ be minimal such that $j \leq i$ and every leaf of $v_j, \ldots, v_{i-1}$ is disconnected from $c \in \Gamma$. The {\bf disconnected region} of $c$, denoted $\DR_c(T)$, is the set of leaves attached to $v_j, \ldots, v_{i-1}$. We let 
\[
m_c = m_c(T) = \min(\DR_c(T))
\]
be its minimum, or $m_c = \infty$ if $\DR_c(T)$ is empty.
\end{defn}

\begin{example}\label{ex:left-right}

In the diagram below, $\DR_2=\{\}$, $\DR_4=\{\}$, $\DR_5=\{2,4\}$, $\DR_1=\{5\}$, $\DR_3=\{\}$, $\DR_6=\{2,4,5\}$.  So, for instance, $m_2=m_4=m_3=\infty$, $m_1=5$, $m_5=2$, and $m_6=2$.

\begin{center}
\begin{asy}
import math;
import graph;
unitsize(1cm);

draw((0,0)--(5,0));

label("$P$",(0,0),W);
label("$Q$",(5,0),E);

draw((1,0)--(1,-0.5));
draw((2,0)--(2,-0.5));
draw((3,0)--(3,-0.5));
draw((3.75,-0.5)--(4,0)--(4.25,-0.5));
draw((4,0)--(4,-0.5));

label("$2$",(1,-0.5),S);
label("$4$",(2,-0.5),S);
label("$5$",(3,-0.5),S);
label("$1$",(3.75,-0.5),S);
label("$3$",(4,-0.5),S);
label("$6$",(4.25,-0.5),S);

draw((6+1,0)--(7+1,0)--(6.5+1,-0.8)--cycle);
draw((8+1,0)--(8.5+1,-0.8));

dot((6+1,0));
dot((7+1,0));
dot((6.5+1,-0.8));
dot((8+1,0));
dot((8.5+1,-0.8));
dot((9+1,0));

label("$1$",(6+1,0),W);
label("$2$",(7+1,0),E);
label("$4$",(6.5+1,-0.8),S);
label("$3$",(8+1,0),W);
label("$5$",(8.5+1,-0.8),S);
label("$6$",(9+1,0),E);

\end{asy} 
\end{center}

\end{example}

Note that $c$ is disconnected in $\Gamma$ from every element in $\DR_c(T)$. 

\begin{remark}\label{rmk:DR}
    Let $v$ be as in Definition \ref{def:disconnected-region} with leaves $c_1 < \cdots < c_k$. If $m_{c_i} < \infty$ for some $i$, there must be a vertex $w$ to the left of $v$. If moreover $m_{c_i} < \infty$ for all $i$, then the leaves attached at $w$ form an independent set with $c_1, \ldots, c_k$.
\end{remark}

\begin{defn}
    For a tree $T\in \millipedes(\Gamma)$, we say a vertex $v$ is a \textbf{bad vertex} if either it has more than one leaf attached (ignoring $P$ and $Q$), or it has a singleton leaf $c$ attached and has $m_c<c$.
\end{defn}

The involution considers the leftmost bad vertex and performs an operation to form a new tree as follows.

\begin{defn}\label{defn:invOSPgamma}
    We define the \textbf{left-right involution}  $\varphi^\leftrightarrows:\millipedes(\Gamma)\to \millipedes(\Gamma)$ as follows.  Let $T\in \millipedes(\Gamma)$.  
    If there are no bad vertices along the path from $P$ to $Q$, define $\varphi^\leftrightarrows(T)=T$.  Otherwise, let $v$ be the leftmost bad vertex along the path from $P$ to $Q$ and let $c_1<\cdots<c_k$ be the leaves attached at $v$.  
    \begin{itemize}
        \item \textbf{Merge case:} If $m_{c_i}<c_i$ for all $i$, there is a singleton leaf $c_0$ attached at the vertex $v_0$ just left of $v$ along the path. We define $\varphi^\leftrightarrows(T)$ to be the tree formed by merging $c_0$ with $c_1,\ldots,c_n$ so they are all attached at $v$ (and deleting vertex $v_0$).
        \item \textbf{Split case:} Otherwise, let $i$ be the smallest index such that $m_{c_i}>c_i$.  Then we remove leaf $c_i$ from attaching at $v$ and create a new vertex $v_0$ just before $v$ that we attach a singleton leaf labeled $c_i$ to, to form $\varphi^\leftrightarrows(T)$.
    \end{itemize}
\end{defn}

\begin{example}
    Continuing Example \ref{ex:left-right}, the leftmost bad vertex is the $5$, and $m_5 < 5$. We obtain, from the merge case:

\begin{center}
\begin{asy}
import math;
import graph;
unitsize(1cm);

draw((0,0)--(5,0));

label("$P$",(0,0),W);
label("$Q$",(5,0),E);

draw((1,0)--(1,-0.5));
draw((2,0)--(2,-0.5));
draw((3,0)--(3,-0.5));
draw((3.75,-0.5)--(4,0)--(4.25,-0.5));
draw((4,0)--(4,-0.5));

label("$2$",(1,-0.5),S);
label("$4$",(2,-0.5),S);
label("$5$",(3,-0.5),S);
label("$1$",(3.75,-0.5),S);
label("$3$",(4,-0.5),S);
label("$6$",(4.25,-0.5),S);

draw((6,-0.25)--(7,-0.25), arrow=Arrows(TeXHead));

currentpicture=shift((-8,0))*currentpicture;

draw((0,0)--(4,0));

label("$P$",(0,0),W);
label("$Q$",(4,0),E);

draw((1,0)--(1,-0.5));
draw((1.75,-0.5)--(2,0)--(2.25,-0.5));
draw((3, -0.5)--(3, 0));
draw((2.75,-0.5)--(3,0)--(3.25,-0.5));

label("$2$",(1,-0.5),S);
label("$4$",(1.75,-0.5),S);
label("$5$",(2.25,-0.5),S);
label("$1$",(2.75,-0.5),S);
label("$3$",(3,-0.5),S);
label("$6$",(3.25,-0.5),S);

\end{asy} 
\end{center}

Note that on the right we have $\DR_2=\{\} = \DR_4$, so the $45$ vertex is the leftmost bad vertex and has $m_4 = \infty$ (the split case).

\end{example}

\begin{example}
With the same graph $\Gamma$ as in the previous example \ref{ex:left-right}, we have:

\begin{center}
\begin{asy}
import math;
import graph;
unitsize(1cm);

draw((0,0)--(5,0));

label("$P$",(0,0),W);
label("$Q$",(5,0),E);

draw((1,0)--(1,-0.5));
draw((2,0)--(2,-0.5));
draw((3,0)--(3,-0.5));
draw((3.75,-0.5)--(4,0)--(4.25,-0.5));
draw((4,0)--(4,-0.5));

label("$5$",(1,-0.5),S);
label("$1$",(2,-0.5),S);
label("$2$",(3,-0.5),S);
label("$3$",(3.75,-0.5),S);
label("$4$",(4,-0.5),S);
label("$6$",(4.25,-0.5),S);

draw((6,-0.25)--(7,-0.25), arrow=Arrows(TeXHead));

currentpicture=shift((-8,0))*currentpicture;

draw((0,0)--(6,0));

label("$P$",(0,0),W);
label("$Q$",(6,0),E);

draw((1,0)--(1,-0.5));
draw((2,0)--(2,-0.5));
draw((3,0)--(3,-0.5));
draw((4,0)--(4,-0.5));
draw((4.75,-0.5)--(5,0)--(5.25,-0.5));

label("$5$",(1,-0.5),S);
label("$1$",(2,-0.5),S);
label("$2$",(3,-0.5),S);
label("$4$",(4,-0.5),S);
label("$3$",(4.75,-0.5),S);
label("$6$",(5.25,-0.5),S);

\end{asy} 
\end{center}
On the left, we have $\DR_5 = \{\}$, $\DR_1 = \{5\}$, $\DR_2 = \{\}$, $\DR_3 = \{1, 2\}$, $\DR_4 = \{\}$ and $\DR_6 = \{1, 2, 5\}$. The leftmost bad vertex is the $346$ and has $m_3 < 3$, $m_4 > 4$ and $m_6 < 6$ (the split case). On the right, we now have $\DR_4 = \{\}, \DR_3 = \{1,2,4\}$ and $\DR_6 = \{1,2,4,5\}$, hence $m_3 < 3$ and $m_6 < 6$ (the merge case). Note that in this example, the entry $4$ that splits and merges is not the minimum of its block.
\end{example}

\begin{remark}[Recovering the merge-split involution $\invOSP$ in Losev-Manin case]\label{rmk:recover-OSP-from-graph-case}
    In the special case where $\Gamma$ has no edges, $\varphi^\leftrightarrows$ recovers the merge-split involution $\invOSP$ on ordered set partitions (Definition \ref{def:invOSP}). To see this, we observe that $\DR_c(T)$ is then the complete set of leaves left of $c$, so the condition $m_c > c$ says $c$ is smaller than every leaf attached to vertices to its left. If the leftmost bad vertex has leaves $c_1 < \cdots < c_k$, the merge and split conditions therefore depend only on $c_1$, where they reduce to those of $\invOSP$.

    Note that this agrees with the Losev-Manin case, because there are two points, the dominant points $P$ and $Q$, that cannot collide with any vertex (two heavy points) and the rest can occupy the same spaces (light points).
\end{remark}

\begin{lemma}\label{lem:well-defined}
The map $\varphi^\leftrightarrows$ is well defined.
\end{lemma}

\begin{proof}
The merge case is well-defined and remains a millipede by Remark \ref{rmk:DR}.  The split case is always well-defined and also forms a millipede output because a subset of an independent set is an independent set. 
\end{proof}

\begin{lemma}\label{lem:involution}
The map $\varphi^\leftrightarrows$ is an involution.
\end{lemma}

\begin{proof}
    Let $T\in \millipedes(\Gamma)$ and $T'=\varphi^\leftrightarrows(T)$.  We wish to show $\varphi^\leftrightarrows(T')=T$.  If $T'=T$ we are done.  Otherwise, let $v$ be the leftmost bad vertex in $T$, with leaves $c_1,\ldots,c_k$ attached.

    \textbf{Case 1 (Merge case):} Suppose $m_{c_i}<c_i$ for all $i=1,\ldots,k$.  Then each disconnected region $\DR_{c_i}$ is nonempty, so the leaf $c_0$ attached to the vertex $w$ just left of $v$ exists and forms an independent set with $c_1,\ldots,c_k$ in the graph. Since we are assuming $v$ is the leftmost bad vertex, we also have $m_{c_0}>c_0$ (otherwise $w$ would be a bad vertex to the left of $v$).  

    Now, by the definition of $\varphi^\leftrightarrows$, $T'$ is formed by merging $c_0$ with the leaves attached at $v$, by moving $c_0$ to be attached at $v$ and deleting the vertex $v_0$ that $c_0$ was attached to.  Suppose the new least to greatest ordering of the vertices is $c_1,c_2,\ldots,c_i,c_0,c_{i+1},\ldots,c_k$.  Then the new minima $m'_{c_j}=m_{c_j}(T')$ are unchanged for $j\le i$ since then $m_{c_j}\le c_j<c_0$, and so $m'_{c_j}=m_{c_j}$. It follows that $c_0$ is the smallest leaf attached to $v$ such that $m_{c_0}>c_0$, so $\varphi^\leftrightarrows(T')$ is formed by splitting off $c_0$ to the left again, returning $T$.

    \textbf{Case 2 (Split case):} Suppose some $m_{c_j}>c_j$, and let $i$ be the smallest index such that $m_{c_i}>c_i$.  Then $T'=\varphi^\leftrightarrows(T)$ is formed by splitting off $c_i$ to be a singleton leaf just to the left of $v$.  Then we still have $m'_{c_i}=m_{c_i}(T')>c_i$ in $T'$ because its disconnected region has not changed.  Thus the new vertex is not bad. 

    Now, consider $m_{c_j}'=m_{c_j}(T')$ for $j=1,\ldots,i-1$ and for $j=i+1,\ldots,k$.  For $j<i$, since $c_i>c_j>m_{c_j}$ we have that $m'_{c_j}=m_{c_j}$ is still less than $c_j$.  For $j>i$, we now have that $c_i<c_j$ is to the left of $c_j$ in its disconnected region, and so $m'_{c_j}<c_j$.  Thus all leaves at $v$ now satisfy $m'_{c_j}<c_j$, so $v$ is again a bad vertex.  Thus to apply $\varphi^\leftrightarrows$ we merge $c_i$ back into the set of leaves attached at $v$ again, and so $\varphi^\leftrightarrows(T')=T$.
\end{proof}

The next lemma is immediate from the definition of $\varphi^\leftrightarrows$.

\begin{lemma}\label{lem:fixed-points}
    The fixed points of $\varphi^\leftrightarrows$ are the trees in $\millipedes(\Gamma)$ that are caterpillars from $P$ to $Q$ with singleton leaf branches in between, such that for every leaf $c\neq 0,\infty$ we have $m_c>c$. 
\end{lemma}

We write $\Fix(\varphi^\leftrightarrows)$ for the set of fixed points of the left-right involution as in the lemma above.

\begin{lemma}\label{lem:fixed-point-bijection}
    There is a bijection $\Fix(\varphi^\leftrightarrows)\to \ACO(\Gamma)$.
\end{lemma}

\begin{proof}
    We define $f:\ACO(\Gamma)\to \Fix(\varphi^\leftrightarrows)$ as follows.  Given an acyclic orientation $\sigma$ of $\Gamma\{0,\infty\}$, first list the largest source, then remove it and list the largest source among the remaining edges, and so on, and make these the singletons in order along the path from $P$ to $Q$ to form a tree $T=f(\sigma)$.  
    
    We check that this map is well defined by showing $T$ satisfies the property of Lemma \ref{lem:fixed-points}.   Suppose $a$ is the largest source at some stage, in the deleted graph $\Gamma'$ formed by deleting the sources encountered before $a$, and $a<c$ with $c$ a largest source at a later stage.  If $c$ is not connected to $a$ or to any $b$ between $a$ and $c$, then $c$ is a larger source than $a$ in $\Gamma'$, a contradiction.  Thus if $a<c$ and $c$ is encountered later in the process, $c$ must be connected in $\Gamma$ to some $b$ weakly right of $a$ and left of $c$.

    We now have a well-defined map $f$.  To construct its inverse, define $g:\Fix(\varphi^\leftrightarrows)\to \ACO(\Gamma)$ by orienting each edge according to the left to right order in $T$ (that is, if $a$ is left of $b$ in $T$ and $a{-}b$ in $\Gamma$ then we orient the edge $a\to b$).  Note that this results in an acyclic graph due to the ordering, so $g$ is well-defined.  Note also that $g\circ f = \mathrm{id}$ by construction.

    Finally, we show $f\circ g = \mathrm{id}$.  Let $T\in \Fix(\varphi^\leftrightarrows)$ and let $\sigma=g(T)$.  Let $a$ be the first singleton in $T$; we claim that $a$ is the largest source of $\sigma$.  Indeed, suppose for contradiction that there were a larger source $c$ than $a$.  Then $c$ occurs later in $T$ than $a$.   Suppose there is some $b$ left of $c$ in $T$ (possibly equal to $a$) that is connected to $c$ in $\Gamma$.  Then by the definition of $\sigma=g(T)$, that edge is oriented towards $c$, contradicting the assumption that $c$ is a source.  Thus there is no such edge, but then $a$ and $c$ contradict the property of Lemma \ref{lem:fixed-points} and $T$ is not a fixed point, a contradiction. 

    Now, removing $a$ from the tree $T$ and the graph $\sigma$ and repeating this process shows that the the next leaf in $T$ is the largest remaining source and so on.  This completes the proof.
\end{proof}

Lemmas \ref{lem:well-defined}, \ref{lem:involution}, \ref{lem:fixed-points}, and \ref{lem:fixed-point-bijection} now together imply Theorem \ref{thm:graphs}.

\begin{remark}
    There is also a non-bijective proof as follows. First, use the $\invheavy$ bijection to reduce to the set of $\Gamma$-millipedes. The sum remaining is
    \[f(\Gamma) := \sum_{(L_1, \ldots, L_r)} (-1)^{r-1}\]
    ranging over ordered set partitions of $V(\Gamma)$ into independent sets. Fixing an edge $e = (v{-}w)$ of $\Gamma$ and writing $\Gamma - e$ and $\Gamma/e$ for the deletion and contraction of $e$, we check that
    \[f(\Gamma - e) = f(\Gamma/e) + f(\Gamma),\]
    which, along with $f(\Gamma \sqcup \{\bullet\}) = -f(\Gamma)$, implies $f(\Gamma) = (-1)^{|\Gamma|-1}\ACO(\Gamma)$. To see the first relation, we view a sequence $(L_1, \ldots, L_r)$ of independent sets of $\Gamma - e$ as either a sequence of independent sets of $\Gamma/e$, if $v$ and $w$ are in the same $L_i$, or as a sequence of independent sets of $\Gamma$ if $v$ and $w$ are in different $L_i$'s.
\end{remark}

\begin{remark} \label{rmk:all-ones-product}
    A priori, the alternating sum of strata does not have a `geometric' meaning, apart from giving the Euler characteristic of the boundary complex. In fact, our up-down involution shows that for many spaces, it can be interpreted as an intersection product, as follows.
\end{remark}

In the following proposition, let $\mathcal{K}$ be a simplicial complex on $[n]=\{1,2,\ldots,n\}$ with $n\ge 1$.  Let $\mathcal{K}\cup\{P,Q,R\}$ denote the simplicial complex formed by adding three isolated points $P,Q,R$.  We similarly write $\mathcal{K}\cup\{P,Q\}$, and so on. Let $\tree(\mathcal{K} \cup \{P, Q\})$ denote the set of $(\mathcal{K} \cup \{P, Q\})$-stable trees. For $T \in \tree(\mathcal{K} \cup \{P, Q\})$, let $i(T)$ be the number of internal edges of $T$. 

We now relate the alternating sum of strata on $\Mbar_{0,\mathcal{K}\cup\{P,Q\}}$ to the product of psi classes $\psi_1\psi_2\cdots\psi_n$ on $\Mbar_{0,\mathcal{K}\cup \{P,Q,R\}}$. 

\begin{prop} \label{prop:all-ones-product}
We have
\[
(-1)^{n-1} \sum_{T \in \tree(\mathcal{K} \cup \{P, Q\})} (-1)^{i(T)}
=
\int_{\Mbar_{0, \mathcal{K} \cup \{P, Q, R\}}}
\psi_1\psi_2\cdots\psi_n.
\]
\end{prop}
Note that, for graphical moduli spaces, the left-hand side is nonnegative by Theorem \ref{thm:graphs}.

\begin{proof}
Let $\OSP([n]; \mathcal{K})$ be the set of ordered set partitions of $[n]$ for which each block is a face of $\mathcal{K}$. Let $\mathfrak{P}_\mathcal{K}$ be the set of unordered set partitions of $\mathcal{K} \cup \{P, Q, R\}$ for which each block is a face of $\mathcal{K}$ (note that $P, Q$ and $R$ each form singleton blocks).

By applying the up-down involution on $\Mbar_{0, \mathcal{K}}$ (see Remark \ref{rmk:up-down-simplicial}), the left-hand side reduces to
\[
(-1)^{n-1}
\sum_{(L_1, \ldots, L_r) \in \OSP([n];\mathcal{K})} (-1)^{r-1}.
\]

On the other hand, the analog of Theorem \ref{thm:psi-formula-graph} for simplicially-stable spaces is
\[
\int_{\Mbar_{0, \mathcal{K}}} \prod_{i=1}^n \psi_i^{k_i} = 
\sum_{\mathcal{P} \in \mathfrak{P}_\mathcal{K}} (-1)^{n+\ell(\mathcal{P})} \int_{\Mbar_{0, \ell(\mathcal{P})}} \prod_{j = 1}^{\ell(\mathcal{P})} \psi_j^{k_{P_j} - |P_j| + 1}.
\]
Specializing $k_i = 0$ for $i = P, Q, R$ and $k_i = 1$ for each other $i$, we have $k_{P_j} = |P_j|$ for all blocks except the three singletons $P, Q, R$. Each of the inner integrals is then of the form
\[
\int_{\Mbar_{0, \ell(\mathcal{P})}} \prod_{j=1}^{\ell(\mathcal{P})-3} \psi_j = (\ell(\mathcal{P})-3)!\ .
\]
Ignoring the singletons $\{P\}, \{Q\}, \{R\}$, this is the number of ways to order the other blocks of $\mathcal{P}$, so
\begin{align*}
\int_{\Mbar_{0, \mathcal{K}}} \prod_{i=1}^n \psi_i^{k_i} 
&= 
\sum_{\mathcal{P} \in \mathfrak{P}_\mathcal{K}} (-1)^{n+\ell(\mathcal{P})} 
(\ell(\mathcal{P})-3)!{} \\
&=\sum_{(L_1, \ldots, L_r) \in \OSP([n];\mathcal{K})} (-1)^{n+r},
\end{align*}
and the claim follows.
\end{proof}
It would be interesting to find an evaluation of the alternating sum of strata for $\Mbar_{0, \mathcal{K}}$ in terms of properties of the simplicial complex $\mathcal{K}$ generalizing the count of acyclic orientations. (See Question \ref{question:simplicial}.)

\section{An SRI for a general multicolored \texorpdfstring{$\psi$}{} class intersection}\label{sec:SRI2}

In this section we give a sign reversing involution to give a positive formula for the general multicolored $\psi$ class intersection numbers.

   We start by introducing new notation for the formula in Lemma \ref{lem:multinomial-form} that enables us to take into account Remark \ref{rmk:coefficients-nonzero} and simplify the sum.  Suppose we are working in $\Mbar_{0,[r_1,\ldots,r_m]}$, so there are $r_1$ marked points of color $C^{(1)}$, $r_2$ marked points of color $C^{(2)}$, and so on, up to $r_m$ marked points of color $C^{(m)}$. Consider an arbitrary product of $\psi$ classes on $\Moduli_{0,[r_1,\dots,r_m]}$:
    \begin{equation}\label{eqn:psi-product}
        \int \left(\prod_{i=1}^{\ell_1}\psi_i^{k_i^{(1)}}\right)\left(\prod_{i=r_1+1}^{r_1+\ell_2}\psi_i^{k_i^{(2)}}\right)\cdots \left(\prod_{i=r_1+\cdots+r_{m-1}+1}^{r_1+\cdots+r_{m-1}+\ell_m}\psi_i^{k_i^{(m)}}\right).
    \end{equation} 
    Throughout this section we assume that for any $j$ we have $$0 < k_1^{(j)}\le k_2^{(j)}\le \cdots \le k_{\ell_j}^{(j)};$$ that is, the positive exponents on the $\psi$ classes in the product for each color are ordered to be increasing and associated with the first $\ell_j$ marked points of that color (which we may assume by symmetry of marked points within each color).
    We sometimes simply write $k_i$ rather than $k_i^{(j)}$, since the $j$ superscript merely indicates the color of $i$. For any $S\subseteq [n]$, we write $k_S := \sum_{i\in S} k_i$.

\begin{example}
    As a running example for this section, we consider the case where we have five colors of marked points, Red ($C^{(1)}=R$), Azure ($C^{(2)}=A$), Green ($C^{(3)}=G$), Yellow ($C^{(4)}=Y$), and Violet ($C^{(5)}=V$). Suppose the first six marked points are red (so $r_1=6$), then the next five are azure ($r_2=5$), the next four are green ($r_3=4$), the next two are yellow ($r_4=2$), and the next two are violet ($r_5=2$), for a total of $n=19$ marked points, forming the moduli space $\Mbar_{0,[6,5,4,2,2]}$. Written in the form of Equation \eqref{eqn:psi-product}, one $\psi$ class intersection product on this space is
    $$\int_{\Moduli_{0,[6,5,4,2,2]}} \psi_1^2\psi_2^3\psi_3^5 \cdot \psi_7^4 \cdot \psi_{18}\psi_{19}.$$
    We have $\ell_1=3$, $\ell_2=1$, $\ell_3=0$, $\ell_4=0$, $\ell_5=2$, since we have three $\psi$ classes on the red numbers, one on the azure, and two on the violet.  We also have $k_R=10$, $k_A=4$, $k_G=k_Y=0$, and $k_V=2$.
\end{example}

\begin{defn}\label{def:hue}
    Let $C$ be a color and consider a set partition of the marked points of
    color $C$ whose exponent in Equation \eqref{eqn:psi-product} is nonzero. We say each block $B$ of this partition is a {\bf hue} of $C$ and denote the hue $C_B$. We say $C_B$ {\bf matches} the color $C$ (and does not match any other color).
\end{defn}

\begin{defn}\label{def:decoration}
  A \textbf{decoration} for the product \eqref{eqn:psi-product} is a way of marking the numbers $1,2,\ldots,n$ with the following properties:
  \begin{itemize}
      \item The smallest number in each of the first three colors ($1, r_1+1$, and $r_1+r_2+1$) is crossed out. 
      The numbers $i$ with $k_i>0$ are boxed.  
      \item Fix a set partition $\mathcal{P}$ of the boxed numbers, with monochromatic blocks, which we call the {\bf hue partition}. 
      For each block $B$ of $\mathcal{P}$, assign the largest $|B|-1$ numbers of $B$ the hue $C_B$ and underline them.
      \item For each block $B$ of $\mathcal{P}$, assign the hue $C_B$ to exactly $k_B-(|B|-1)$ of the remaining numbers (that are not crossed out or assigned a hue). Thus a total of $k_B$ numbers have hue $C_B$.
      \item Underline some subset of the unboxed numbers whose hue matches their color.
  \end{itemize}
We write $u(D)$ for the count of underlined numbers of the decoration $D$, and we say the decoration has {\bf sign} $(-1)^{u(D)}$.  Letting $\mathbf{k}$ denote the sequence of exponents in \eqref{eqn:psi-product}, we write $\mathcal{D}(\mathbf{k})$ for the set of all decorations with parameters $\mathbf{k}$.
\end{defn}

Our sign-reversing involutions are defined on these decorations. Since the sign $(-1)^{u(D)}$ comes from the underlined entries, the involutions consist of modifying which numbers are underlined. The first involution deals with the \emph{unboxed} underlined numbers, i.e., those corresponding to marked points whose $\psi$ class does not appear in the product \eqref{eqn:psi-product}, while the second deals with the \emph{boxed} underlined numbers, which correspond to the $\psi$ classes appearing in \eqref{eqn:psi-product}. The fixed points, then, are a subset of the decorations with no underlinings. Note that this forces all of the hues in a fixed point to be singleton blocks.

When visualizing a decoration, we sort the numbers into rows by color, in ascending order, and write the hue assignments as subscripts, as in the following example.

\begin{example} \label{ex:decoration}
Consider our running example $\int \psi_1^2\psi_2^3\psi_3^5\cdot \psi_7^4\cdot \psi_{18}\psi_{19}$ on $\Mbar_{0,[6,5,4,2,2]}$. One valid decoration is shown below, with hue partition $\{1,3\},\{2\},\{7\},\{18,19\}$. The block $\{1,3\}$ forces us to assign $3$ the hue $R_{1,3}$ (and to underline $3$), and similarly for $19 \in \{18, 19\}$ (whose hue $V_{18,19}$ we write as $V$). In total, we must assign three numbers the hue $R_2$, seven the hue $R_{1,3}$ (including the forced subscript at $3$), four the hue $A_7$ (which we write below as $A$), and two the hue $V$ (including the forced subscript on $19$). A sample assignment is shown below. Finally, among the unboxed numbers whose hue matches their color (5, 6, 9, 11), we have selected some to underline.

\begin{center}
\begin{tabular}{cccccccc}
    $R$ & & $\boxed{\cancel{\phantom{.}{\color{red}1}\phantom{.}}}$ & $\boxed{\phantom{.}{\color{red}2}_{A}\phantom{.}}$ & $\boxed{\phantom{.}{{\color{red}\underline{3}}_{R_{1,3}}\phantom{.}}}$ & ${\color{red}4}_V$ & ${\color{red}5}_{R_{1,3}}$ & ${\color{red}\underline{6}}_{R_{2}}$ \\
    
    $A$ & & $\boxed{\cancel{\phantom{.}\color{blue}7\phantom{.}}}$ & ${\color{blue}8}_{R_{1,3}}$ & ${\color{blue}\underline{9}}_A$ & ${\color{blue}10}_{R_{1,3}}$ & ${\color{blue}\underline{11}}_A$ & \\
    
 $G$ & &    $\cancel{{\color{mygreen}12}}$ & ${\color{mygreen}13}_{R_2}$ & ${\color{mygreen}14}_A$ & ${\color{mygreen}15}_{R_{1,3}}$ & & \\
    
    $Y$ & & ${\color{y1}16}_{R_{2}}$ & ${\color{y1}17}_{R_{1,3}}$ & & & & \\
    
 $V$ & &  $\boxed{{\color{y2}18}_{R_{1,3}}}$ & $\boxed{{\color{y2}\underline{19}}_{V}}$ & & & &
\end{tabular}
\end{center}
\end{example}

We first show that decorations essentially directly encode Lemma \ref{lem:multinomial-form}, which we recall again here:

\begin{lemmnf}
Let $\Mbar_{0,[r_1,\dots,r_m]}$ be a multicolored space with $r_1 + \cdots + r_m = n$. We have
$$\int_{\Mbar_{0,[r_1,\ldots,r_m]}}\psi_1^{k_1}\psi_2^{k_2}\cdots \psi_n^{k_n}=\sum_{\mathcal{P}}(-1)^{\sum_{B\in \mathcal{P}}(|B|-1)}\binom{\ell(\mathcal{P})-3}{k_{B_1}-|B_1|+1,k_{B_2}-|B_2|+1,\ldots}$$    
where the sum is over all set partitions $\mathcal{P}$ of the marked points with monochromatic blocks. Recall that $k_B=\sum_{i\in B}k_i$, and we assume that when some $k_{B_j}-|B_j|+1$ is negative, the entire multinomial coefficient is defined to be $0$.
\end{lemmnf}

The set partitions in Lemma \ref{lem:multinomial-form} are related to hue partitions as follows.

\begin{defn}
    Given a decoration $D$, define the \textbf{extended hue partition} $\mathcal{P}_D$ of $\{1,2,\ldots,n\}$ as follows. For each block $B$ in the hue partition of the $\psi$ classes within a color $C^{(j)}$ of $D$, we associate the extended block $$B_D:=B\cup \{t: t\text{ has color } C^{(j)}, \text{ hue }C_B^{(j)},\text{ and is underlined in }D\},$$ and all other numbers in $\{1,2,\ldots,n\}$ are taken to be singleton blocks. 
\end{defn}

\begin{example}
    In the example above, the extended hue partition $\mathcal{P}_D$ has blocks $\{1,3\}$, $\{2,6\}$, $\{7,9,11\}$, $\{18,19\}$, and all other numbers are singleton blocks.
\end{example}

\begin{lemma}\label{lem:alternating-1}
    The $\psi$ class product \eqref{eqn:psi-product} is equal to $\sum_{D\in \mathcal{D}(\bf k)} (-1)^{u(D)}$.
\end{lemma}

\begin{proof}
For each set partition $\mathcal{P}$ containing only monochromatic blocks $B$ for which $k_B-|B|+1\ge 0$ (which happens in any partition associated to a decoration by definition), we break the sum up according to the decorations $D$ for which $\mathcal{P}_D=\mathcal{P}$.  Write $\mathcal{D}(\mathcal{P})$ for the set of all such decorations.  Then the summation becomes \begin{align*}
    \sum_{D\in \mathcal{D}(\bf k)} (-1)^{u(D)}&=\sum_{\mathcal{P}}\sum_{D\in \mathcal{D}(\mathcal{P})} (-1)^{u(D)} \\
    &=\sum_{\mathcal{P}}\sum_{D\in\mathcal{D}(\mathcal{P})}(-1)^{\sum_{B\in \mathcal{P}} |B|-1} \\
    &=\sum_{\mathcal{P}}(-1)^{\sum_{B\in \mathcal{P}}|B|-1} |\mathcal{D}(\mathcal{P})|
\end{align*}
where the second equality follows from the definition of the underlined numbers and of $\mathcal{P}_D$, and the third equality is because the sign we are summing only depends on $\mathcal{P}$, not $D$.

Finally, we compute $|\mathcal{D}(\mathcal{P})|$.  The set partition $\mathcal{P}$ determines all the underlined numbers and their hues; there is therefore one number per block of $\mathcal{P}$ remaining with no hue, apart from the three crossed out numbers.  Thus there are $\ell(\mathcal{P})-3$ numbers remaining to assign a hue, and we have to assign $k_B-(|B|-1)$ more of each hue $C_B$. There are therefore $|\mathcal{D}(\mathcal{P})|=\binom{\ell(\mathcal{P})-3}{k_{B_1}-|B_1|+1,k_{B_2}-|B_2|+1,\ldots}$ ways to assign the remaining hues, so the summation matches the right hand side of Lemma \ref{lem:multinomial-form}.
\end{proof}

We now cancel terms in this alternating sum in two stages, with two sign reversing involutions. 

\subsection{First sign reversing involution \texorpdfstring{$\underline{\varphi}$}{}}

\begin{defn}
    For a fixed sequence of exponents $\mathbf{k}$, define an involution $\underline{\varphi}:\mathcal{D}(\mathbf{k}) \to \mathcal{D}(\mathbf{k})$ as follows.
    Let $m$ be the smallest unboxed number in $\{1,2,\dots,n\}$ whose hue matches its color. Then $\underline{\varphi}(D)$ is the decoration given by toggling whether $m$ is underlined or not in $D$. If no such $m$ exists, then $\underline{\varphi}(D) = D$.
\end{defn}

\begin{example}
    In the decoration of Example \ref{ex:decoration}, $\underline{\varphi}$ underlines the 5 and produces the output:
\begin{center}
    \begin{tabular}{cccccccc}
    $R$ & & $\boxed{\cancel{\phantom{.}{\color{red}1}\phantom{.}}}$ & $\boxed{\phantom{.}{\color{red}2}_{A}\phantom{.}}$ & $\boxed{{\phantom{.}{\color{red}\underline{3}}_{R_{1,3}}}\phantom{.}}$ & ${\color{red}4}_{V}$ & ${\color{red}\underline{5}}_{R_{1,3}}$ & ${\color{red}\underline{6}}_{R_{2}}$ \\
    
    $A$ & & $\boxed{\cancel{\phantom{.}\color{blue}7\phantom{.}}}$ & ${\color{blue}8}_{R_{1,3}}$ & ${\color{blue}\underline{9}}_A$ & ${\color{blue}10}_{R_{1,3}}$ & ${\color{blue}\underline{11}}_A$ & \\
    
 $G$ & &    $\cancel{{\color{mygreen}12}}$ & ${\color{mygreen}13}_{R_2}$ & ${\color{mygreen}14}_A$ & ${\color{mygreen}15}_{R_{1,3}}$ & & \\
    
    $Y$ & & ${\color{y1}16}_{R_{2}}$ & ${\color{y1}17}_{R_{1,3}}$ & & & & \\
    
 $V$ & &  $\boxed{{\color{y2}18}_{R_{1,3}}}$ & $\boxed{{\color{y2}\underline{19}}_{V}}$ & & & &
\end{tabular}
\end{center}
\end{example}

By definition, the fixed points of $\underline{\varphi}$ are precisely those in which every unboxed number has a hue that does not match its color.

\begin{defn}
    We say a decoration is \textbf{mismatched} if every unboxed number has a hue that does not match its color,
    and we write $\mathcal{MD}(\mathbf{k})$ for the set of all mismatched decorations for exponent sequence $\mathbf{k}$.
\end{defn}

Note that in a mismatched decoration, the unboxed numbers are never underlined, but the decoration may still have a negative sign because of underlinings among the boxed numbers. Note also that the extended hue partition for a mismatched decoration is equivalent to the hue partition (treating every unboxed number as an additional singleton block).

\begin{lemma}\label{lem:alternating-2}
       The $\psi$ class product \eqref{eqn:psi-product} is equal to the reduced summation $\sum_{D\in \mathcal{MD}(\mathbf{k})} (-1)^{u(D)}$.
\end{lemma}

\begin{proof}
    Starting from Lemma \ref{lem:alternating-1}, the product is equal to the sum $\sum_{D\in \mathcal{D}(\mathbf{k})} (-1)^{u(D)}$.  The involution $\underline{\varphi}$ is sign reversing because it changes the count of underlined numbers by $\pm 1$. It pairs up cancelling pairs of $+1$s and $-1$s and leaves only the summation over the fixed points of $\underline{\varphi}$, namely the mismatched decorations, completing the proof.
\end{proof}

\subsection{Second sign reversing involution \texorpdfstring{$\phihue$}{}}

We now define a sign reversing involution on mismatched decorations.  The idea is to toggle underlined letters among the boxed letters, which will involve splitting and merging hue blocks. 
In this subsection, we assume all decorations are mismatched.

\begin{notation}
    Throughout this section, we will often write a hue $C_B$ as $B$ when it is clear what color it is.
\end{notation}

\begin{defn}
    In a decoration $D$, we say that a singleton hue $\{t\}$ \textbf{wants to merge with} a hue $B_t$ of the same color if $r=\min(B_t)>t$ and $\{t\}$ is the hue of the number $r$. 
\end{defn}

 We now develop definitions that give conditions for when $\{t\}$ can merge with $B_t$.  To do so, we first define two sequences similar to cycle decompositions of permutations.  We use these orderings to define $\phihue$ in a similar manner to the crucial step of converting from a linear sequence to a product of cycles in Joyal's proof of Cayley's formula for tree enumeration \cite{Joyal}.
\begin{defn}
    Let $D$ be a mismatched decoration and $C$ a color. Let $I$ denote the set of boxed numbers of $C$ excluding any that is crossed out, underlined, or whose hue does not match its color. The {\bf number chain decomposition} in color $C$ in $D$ is the following permutation of $I$.

    Start by writing an open parenthesis followed by $i_1 := \min(I)$. For $k\geq 1$, if $B$ is the hue of $i_k$ with $b=\min(B)$, define
    \begin{equation}\label{eqn:number_chain_defn}
    i_{k+1} :=
    \begin{cases}
    b, & \text{if } b \in I \setminus \{i_1, \ldots, i_k\}, \\
    \min(I \setminus \{i_1, \ldots, i_k\}), & \text{otherwise.}
    \end{cases}
    \end{equation}
    In the first case, write down $i_{k+1}$ and continue extending the chain; in the second case, write a close parenthesis, then an open parenthesis, and write down $i_{k+1}$. Continue until all of $I$ has been listed, and end with a close parenthesis.
    
    The {\bf hue chain decomposition} in color $C$ is the sequence $B_1, \ldots, B_\ell$, where $B_k$ is the hue of $i_k$, with the same parentheses as in the number chain decomposition.  It is thus a permutation of the multiset of hues in $I$.  
\end{defn}

We note that this definition makes sense for arbitrary assignments of sets of numbers within a color to the boxed numbers of that color, not necessarily satisfying the axioms of mismatched decorations. In the proofs below, we will sometimes refer to number and hue chain decompositions of such assignments.

\begin{example}\label{ex:hue-chain}
    Consider the following mismatched decoration for the product $$\int_{\Moduli_{0,[9,4,4]}}\psi_{1}\psi_2\psi_3\psi_4\psi_5\psi_6\psi_7^2\psi_8^3\cdot \psi_{10}^3$$
    and hue partition given by $\{5, 6\}$ and all other hues singletons:
    \begin{center}
       \begin{tabular}{ccccccccccc}
    $R$ & & $\boxed{\cancel{\phantom{.}{\color{red}1}\phantom{.}}}$ & $\boxed{\phantom{.}{\color{red}2}_{R_{5,6}}\phantom{.}}$ & $\boxed{{\phantom{.}{\color{red}3}_{R_7}}\phantom{.}}$ & $\boxed{{\color{red}4}_{R_{8}}}$ & $\boxed{{\color{red}5}_{R_{1}}}$ & $\boxed{{\color{red}\underline{6}}_{R_{5,6}}}$ & $\boxed{{\color{red} 7}_{A_{10}}}$ & $\boxed{{\color{red} 8}_{R_{7}}}$ & ${\color{red} 9}_{A_{10}}$\\
    
    $A$ & & $\boxed{\cancel{\phantom{.}\color{blue}10\phantom{.}}}$  & ${\color{blue}11}_{R_8}$ &${\color{blue}12}_{R_3}$ &${\color{blue}13}_{R_8}$ & & \\
    
 $G$ & &    $\cancel{{\color{mygreen}14}}$  & ${\color{mygreen}15}_{R_2}$ & ${\color{mygreen}16}_{A_{10}}$ & ${\color{mygreen}17}_{R_4}$\\
\end{tabular}
\end{center}
The number chain decomposition for the red numbers is $$(2,5),(3),(4,8)$$
and the corresponding hue chain decomposition is 
$$(R_{5,6},R_1),(R_7),(R_8,R_7),$$
which we may also rewrite in terms of blocks as
$$(\{5,6\},\{1\}),(\{7\}),(\{8\},\{7\}).$$
\end{example}

We record two observations about number and hue chain decompositions.

\begin{lemma}\label{lem:hue-facts} Consider a mismatched decoration.
\begin{enumerate}
    \item In the number chain decomposition in a color $C$, if $i$ is the start of a new chain, then every element of $I$ less than $i$ occurs earlier in the sequence.

    \item In the hue chain decomposition, all occurrences a hue $B$ of color $C$ after the first are ends of chains.
\end{enumerate}
\end{lemma}
\begin{proof}
    Statement (1) corresponds to the second case in Equation \eqref{eqn:number_chain_defn}.
    For statement (2), the first time $B$ occurs in the hue chain decomposition, its minimum $b = \min(B)$ is next in the number chain, unless it has already been listed. Either way, any further occurrences of $B$ lead into the second case of Equation \eqref{eqn:number_chain_defn}.
\end{proof}

We now give conditions for merging hues.

\begin{defn}
    In a mismatched decoration, the \textbf{permission list} of a set of numbers $S$ whose hues all have color $C$ is the ordering of the elements of $S$ as follows: first list all elements of $S$ that are in color $C$ in number chain order, and then list the remaining elements of $S$ in ascending order.  We call this ordering of $S$ the \textbf{permission order}.
\end{defn}

\begin{example}
  In Example \ref{ex:hue-chain}, the permission list of the set of all numbers with hues $R_1$, $R_7$, or $R_8$ is $5,3,4,8,11,13$.    
\end{example}

\begin{defn}\label{def:permission}
    Suppose $\{t\}$ wants to merge with block $B_t$, with $r=\min(B_t)$.  Then $\{t\}$ \textbf{has permission} to merge with $B_t$ if the following conditions are satisfied:
    \begin{enumerate}
    \item[(i)] In the number chain decomposition, $r$ is the first number of hue $\{t\}$.
    
        \item[(ii)] If $r$ is not the start of its number chain, let $q$ be the number preceding it.  Consider the set $S$ of all non-underlined numbers  with hue $\{t\}$ or $B_t$ besides $q$ (if it exists) and besides $r$.  In the permission list of $S$, those with hue $B_t$ all precede those with hue $\{t\}$.
    \end{enumerate}
\end{defn}

\begin{defn}
    If $q$ exists in condition (ii) above, we say $\{t\},B_t$ are in \textbf{Situation Q}, and otherwise we say they are in \textbf{Situation \nQ}.
\end{defn}

\begin{remark}
    In Situation Q, the hue of $q$ is $B_t$ by the definition of the number chain.
\end{remark}

Condition (i) immediately implies the following lemma.

\begin{lemma}\label{lem:clinging-unique}
    If $\{t\}$ wants to merge with multiple blocks, it can only have permission to merge with at most one of them.
\end{lemma}

\begin{example}
    In Example \ref{ex:hue-chain}, hue $\{1\}$ wants to merge with $\{5,6\}$ and has $q=2$, $r=5$; no other numbers with hue $\{1\}$ or $\{5,6\}$ occur, so $\{1\}$ has permission to merge with $\{5,6\}$. In addition, $\{7\}$ wants to merge with $\{8\}$; however, $8$ is not the first number with hue $\{7\}$ in the number chain decomposition, so $\{7\}$ does not have permission to merge with $\{8\}$. 
\end{example}

\begin{example}\label{ex:mod-hue-chain}
    If we modify Example \ref{ex:hue-chain} to the following:    \begin{center}
       \begin{tabular}{ccccccccccc}
    $R$ & & $\boxed{\cancel{\phantom{.}{\color{red}1}\phantom{.}}}$ & $\boxed{\phantom{.}{\color{red}2}_{R_{5,6}}\phantom{.}}$ & $\boxed{{\phantom{.}{\color{red}3}_{R_8}}\phantom{.}}$ & $\boxed{{\color{red}4}_{R_{8}}}$ & $\boxed{{\color{red}5}_{R_{1}}}$ & $\boxed{{\color{red}\underline{6}}_{R_{5,6}}}$ & $\boxed{{\color{red} 7}_{A_{10}}}$ & $\boxed{{\color{red} 8}_{R_{7}}}$ & ${\color{red} 9}_{A_{10}}$\\
    
    $A$ & & $\boxed{\cancel{\phantom{.}\color{blue}10\phantom{.}}}$  & ${\color{blue}11}_{R_8}$ &${\color{blue}12}_{R_3}$ &${\color{blue}13}_{R_7}$ & & \\
    
 $G$ & &    $\cancel{{\color{mygreen}14}}$  & ${\color{mygreen}15}_{R_2}$ & ${\color{mygreen}16}_{A_{10}}$ & ${\color{mygreen}17}_{R_4}$ \\
\end{tabular}
\end{center}
then the number chain decomposition becomes:
$$(2,5),(3,8),(4),$$
with hue chain decomposition $$(\{5,6\},\{1\}),(\{8\},\{7\}),(\{8\}).$$
The hue $\{7\}$ still wants to merge with $\{8\}$, and now $r=8$ is the first number in number chain decomposition order with hue $\{7\}$.  Furthermore, $q=3$ for the candidate hue $\{7\}$, and the remaining numbers with hue $\{7\}$ or $\{8\}$, in permission order, are $4,11,13$ with respective hues $\{8\},\{8\},\{7\}$. Thus $\{7\}$ has permission to merge with $\{8\}$.

The hue $\{7\}$ is not the minimal $\{t\}$; that is still $\{1\}$, which has permission to merge with $\{5,6\}$.  For $t=1$, we have $q=2$. 
\end{example}

\begin{defn}[The involution $\phihue$]\label{def:involution-2}
 Given $D\in \mathcal{MD}(\mathbf{k})$, let $B$ (if it exists) be the non-singleton hue with the smallest minimal element $m=\min(B)$.  Let $\{t\}$ (if it exists) be the smallest singleton hue with permission to merge with some (necessarily unique) block $B_t$.

    We define the map $\varphi^{\mathrm{hue}} :\mathcal{MD}(\mathbf{k})\to \mathcal{MD}(\mathbf{k})$ by the following three cases:
\begin{itemize}
    \item  \textbf{Split:} If $m<t$ or $t$ does not exist, split $B$ into two blocks, $\{m\}$ and $B'=B\setminus \{m\}$, and let $r=\min(B')$.  
   If there is a number chain starting with some $i<r$ containing a number of hue $B$, let $q$ be the first number of hue $B$ in number chain order.  Redecorate as follows:
    \begin{enumerate}
        \item Relabel the hue $B$ on $r=\min(B')$ as $\{m\}$ and remove its underline.
        \item Relabel the hues $B$ as $B'$ on the remaining numbers in $B'$.
        \item If $q$ exists, relabel its hue from $B$ to $B'$. 
        
        \item Let $S$ be the set of remaining numbers with hue $B$ in the decoration.  In the permission list order of $S$, relabel the hues first with $B'$ until there are a total of $k_{B'}$ numbers with hue $B'$ in the decoration, and then relabel the rest $\{m\}$.
    \end{enumerate}
     \item \textbf{Merge:} If $t<m$ or $m$ does not exist, merge the blocks $\{t\}$ and $B_t$ into a block $B=\{t\}\cup B_t$, relabel all hues $\{t\}$ and $B_t$ as $B$, and underline the minimum element of $B_t$. 
     \item \textbf{Do nothing:} If neither $t$ nor $m$ exists, do nothing.
\end{itemize}
We define the output to be $\varphi^{\mathrm{hue}}(D)$.
\end{defn}

\begin{remark}
    Python code implementing mismatched decorations and the hue involution is available at \url{https://github.com/jakelev/psi_decorations}.  As a sanity check, we have used this code to verify that $\phihue$ indeed acts as a sign-reversing involution on the mismatched decorations for every product on multicolored spaces with $\leq 11$ marked points (approx. 500 million mismatched decorations).
\end{remark}

\begin{remark}
    Given a decoration $D$, if the algorithm for computing $\phihue(D)$ follows the split procedure, we say $D$ is in the \textbf{split case}.  If it follows the merge procedure, we say $D$ is in the \textbf{merge case}.  Otherwise, we say $D$ is a \textbf{fixed point}.
\end{remark}

\begin{defn}
We say that a decoration in the split case is in \textbf{Situation Q} if $q$ exists as in Definition \ref{def:involution-2}, and in \textbf{Situation \nQ} if it does not.  
\end{defn}

We will see that the property of being in Situation Q or not is preserved by $\phihue$.  

\begin{example}
We examine the decoration in Example \ref{ex:mod-hue-chain}, which is a decoration for the $\psi$ class product $$\int_{\Moduli_{0,[9,4,4]}}\psi_{1}\psi_2\psi_3\psi_4\psi_5\psi_6\psi_7^2\psi_8^3\cdot \psi_{10}^3,$$
starting with 
\begin{center}
       \begin{tabular}{ccccccccccc}
    $R$ & & $\boxed{\cancel{\phantom{.}{\color{red}1}\phantom{.}}}$ & $\boxed{\phantom{.}{\color{red}2}_{R_{5,6}}\phantom{.}}$ & $\boxed{{\phantom{.}{\color{red}3}_{R_8}}\phantom{.}}$ & $\boxed{{\color{red}4}_{R_{8}}}$ & $\boxed{{\color{red}5}_{R_{1}}}$ & $\boxed{{\color{red}\underline{6}}_{R_{5,6}}}$ & $\boxed{{\color{red} 7}_{A_{10}}}$ & $\boxed{{\color{red} 8}_{R_{7}}}$ & ${\color{red} 9}_{A_{10}}$\\
    
    $A$ & & $\boxed{\cancel{\phantom{.}\color{blue}10\phantom{.}}}$  & ${\color{blue}11}_{R_8}$ &${\color{blue}12}_{R_3}$ &${\color{blue}13}_{R_7}$ & & \\
    
 $G$ & &    $\cancel{{\color{mygreen}14}}$  & ${\color{mygreen}15}_{R_2}$ & ${\color{mygreen}16}_{A_{10}}$& ${\color{mygreen}17}_{R_4}$\\
\end{tabular}
\end{center}
we have $m=5$ and $t=1$, so we are in the merge case, and we are in Situation Q with $q=2$. The output of $\phihue$ is

\begin{center}
       \begin{tabular}{ccccccccccc}
    $R$ & & $\boxed{\cancel{\phantom{.}{\color{red}1}\phantom{.}}}$ & $\boxed{\phantom{.}{\color{red}2}_{R_{1,5,6}}\phantom{.}}$ & $\boxed{{\phantom{.}{\color{red}3}_{R_8}}\phantom{.}}$ & $\boxed{{\color{red}4}_{R_{8}}}$ & $\boxed{{\color{red}\underline{5}}_{R_{1,5,6}}}$ & $\boxed{{\color{red}\underline{6}}_{R_{1,5,6}}}$ & $\boxed{{\color{red} 7}_{A_{10}}}$ & $\boxed{{\color{red} 8}_{R_{7}}}$ & ${\color{red} 9}_{A_{10}}$\\
    
    $A$ & & $\boxed{\cancel{\phantom{.}\color{blue}10\phantom{.}}}$  & ${\color{blue}11}_{R_8}$ &${\color{blue}12}_{R_3}$ &${\color{blue}13}_{R_7}$ & & \\
    
 $G$ & &    $\cancel{{\color{mygreen}14}}$  & ${\color{mygreen}15}_{R_2}$ & ${\color{mygreen}16}_{A_{10}}$ & ${\color{mygreen}17}_{R_4}$ \\
\end{tabular}
\end{center}
This new decoration is now in the split case with $m=1$.  The new number chain decomposition of the red row is $$(2),(3,8),(4),$$ which differs from the previous number chain decomposition by simply deleting the $5$.  It is furthermore in Situation Q because $r=5$ and $i=2<5$ starts a chain with a number $q$ of hue $\{1,5,6\}$, namely $q=2$.

Another application of $\phihue$ splits the block $\{1,5,6\}$ into two blocks $\{1\}$ and $B'=\{5,6\}$, labels $r=5$ as $\{1\}$ and removes its underline (Step 1), relabels $6$ as $\{5,6\}$ and keeps its underline (Step 2), and relabels $q=2$ as $\{5,6\}$ by Step 3. Step 4 is vacuous in this case, and we see that applying $\phihue$ twice returns us to the original decoration. 
\end{example}

\begin{example}
    We now give an example of the involution in Situation NQ.  Let $D$ be the decoration:
    
     \begin{center}
       \begin{tabular}{ccccccccccc}
    $R$ & & $\cancel{\phantom{.}{\color{red}1}\phantom{.}}$ & ${\color{red}2}_{\{7,8\}}$ & &  & &  & &  & \\
    
    $A$ & & $\boxed{\cancel{\phantom{.}\color{blue}3\phantom{.}}}$  & $\boxed{{\color{blue}4}_{\{6\}}}$ & $\boxed{{\color{blue}5}_{\{9\}}}$ & $\boxed{{\color{blue}6}_{\{9\}}}$ & $\boxed{{\color{blue}7}_{\{6\}}}$ & $\boxed{{\color{blue}\underline{8}}_{\{7,8\}}}$ & $\boxed{{\color{blue}9}_{\{5\}}}$ & $\boxed{{\color{blue}10}_{\{7,8\}}}$ \\
    
 $G$ & &    $\cancel{{\color{mygreen}11}}$  & ${\color{mygreen}12}_{\{5\}}$ & ${\color{mygreen}13}_{\{3\}}$ & ${\color{mygreen}14}_{\{10\}}$ & ${\color{mygreen}15}_{\{4\}}$ \\

 $V$ & & ${\color{violet} 16}_{\{7,8\}}$ &  ${\color{violet} 17}_{\{10\}}$ & ${\color{violet} 18}_{\{10\}}$ & ${\color{violet} 19}_{\{9\}}$ & ${\color{violet} 20}_{\{7,8\}}$
\end{tabular}
\end{center} 
    This is an example of a mismatched decoration for the product $$\int_{\Mbar_{0,[2,8,5,5]}} \psi_3 \psi_4\psi_5^2\psi_6^2\psi_7^2\psi_8^3\psi_9^3\psi_{10}^3.$$  Notice that the only $\psi$ classes are on Azure points (and the exponents are increasing accordingly), and so rather than writing hues like $A_{7,8}$ we simply write the hue as its block $\{7,8\}$ for the subscripts above.  
    
    Above, the hues that want to merge with another block are $\{5\}$, which wants to merge with $\{9\}$, and $\{6\}$, which wants to merge with $\{7\}$.  The number chain and hue chain decompositions of the non-underlined, non-crossed out azure boxed letters are:
    $$(\,\,4\,\,,\,\,6\,\,,\,\,9\,\,,\,\,5\,\,),(\,\,7\,\,),(\,\,10\,\,)$$
    $$(\{6\},\{9\},\{5\},\{9\}),(\{6\}),(\{7,8\})$$
    Since the number $7$ is not the first number with hue $\{6\}$ in number chain order, $\{6\}$ does not have permission to merge with $\{7,8\}$ by condition (i) for permission.  On the other hand, while $\{5\}$ and $\{9\}$ satisfy condition (i) for permission to merge, they do not satisfy condition (ii), because the permission list order for the remaining numbers in the entire decoration gives the ordering of hues $\{9\},\{5\},\{9\}$ in that order, since number 12 has hue $\{5\}$ and number 19 has hue $\{9\}$.  Therefore $\{5\}$ does not have permission to merge with $\{9\}$ either.  
    
    So $D$ is in the split case with its only nonsingleton block, $\{7,8\}$.  To split, for Step 1 we remove the underline on $8$ and rename its hue to $\{7\}$; in the chain decompositions, this inserts a number chain $(8)$, with hue $(\{7\})$, after the chain $(7)$, so we are indeed in Situation NQ.  Step 2 is vacuous since $\{8\}$ is a singleton hue now, and Step 3 does not occur since we are in Situation NQ.  For Step 4, we relabel the hue of $10$ as $\{8\}$ and then continue to relabel the hues of $2$, $16$, $20$ so that there are three numbers with hue $\{8\}$ and two with hue $\{7\}$ in the decoration.  The result is as follows:  

         \begin{center}
       \begin{tabular}{ccccccccccc}
    $R$ & & $\cancel{\phantom{.}{\color{red}1}\phantom{.}}$ & ${\color{red}2}_{\{8\}}$ & &  & &  & &  & \\
    
    $A$ & & $\boxed{\cancel{\phantom{.}\color{blue}3\phantom{.}}}$  & $\boxed{{\color{blue}4}_{\{6\}}}$ & $\boxed{{\color{blue}5}_{\{9\}}}$ & $\boxed{{\color{blue}6}_{\{9\}}}$ & $\boxed{{\color{blue}7}_{\{6\}}}$ & $\boxed{{\color{blue}{8}}_{\{7\}}}$ & $\boxed{{\color{blue}9}_{\{5\}}}$ & $\boxed{{\color{blue}10}_{\{8\}}}$ \\
    
 $G$ & &    $\cancel{{\color{mygreen}11}}$  & ${\color{mygreen}12}_{\{5\}}$ & ${\color{mygreen}13}_{\{3\}}$ & ${\color{mygreen}14}_{\{10\}}$ & ${\color{mygreen}15}_{\{4\}}$ \\

 $V$ & & ${\color{violet} 16}_{\{8\}}$ &  ${\color{violet} 17}_{\{10\}}$ & ${\color{violet} 18}_{\{10\}}$ & ${\color{violet} 19}_{\{9\}}$ & ${\color{violet} 20}_{\{7\}}$
\end{tabular}
\end{center} 
where the new number and hue chain decompositions are: 
    $$(\,\,4\,\,,\,\,6\,\,,\,\,9\,\,,\,\,5\,\,),(\,\,7\,\,),(\,\,8\,\,),(\,\,10\,\,)$$
    $$(\{6\},\{9\},\{5\},\{9\}),(\{6\}),(\{7\}),(\{8\})$$
    Note that the same reasoning as above shows that, in this new output, $\{6\}$ does not have permission to merge with $\{7\}$ and $\{5\}$ does not have permission to merge with $\{9\}$.  But $\{7\}$ now has permission to merge with $\{8\}$, and is Situation NQ for permission, with the remaining hues after number $r=8$ in permission list order being $\{8\},\{8\},\{8\},\{7\}$.  Thus applying $\phihue$ again returns us to the original decoration.
\end{example}

\begin{lemma}
    The map $\phihue$ is well-defined.
\end{lemma}

\begin{proof}
   Given a mismatched decoration $D$, it either has an $m$ or $t$ as in the definition or does not; if it does not, $\phihue(D) = D \in \mathcal{MD}(\mathbf{k})$. If $D$ has an $m$ or $t$, we consider two cases.

   \textbf{Split case:} If $D$ is in the split case, then it has an $m$ with $m<t$ for $\{t\}$ with permission to merge. We first show that all the steps of the split algorithm are well-defined.  For Step 1, note that we must have $k_m\ge 1$, because $m$ is boxed as a number, so we have a hue $m$ available for this step. We also only need to underline the largest $|B'|-1$ elements of $B'$ (and ensure that they have hue $B'$), so we may remove the underline on $r$ as well. Step 2 is clearly well-defined since there are at least $k_{B'}\ge |B'|$ copies of hue $B'$ available for the new decoration. For Step 3, again since $k_{B'}\ge |B'|$, and only $|B'|-1$ elements are labeled $B'$ so far (in Step 2), there is an extra $B'$ hue available for $q$. Then, Step 4 is clearly well-defined, and by its definition, after Step 4 we have no more hues labeled $B$.
   
   Moreover, Step 4 guarantees that we have the correct number of each hue as given by $\mathbf{k}$ with the new set partition. The result $\phihue(D)$ is also still mismatched because $D$ was mismatched, and we are only changing the hues on some blocks to different hues of the same color. Thus the output of the split case is a well-defined mismatched decoration for the same sequence $\mathbf{k}$.

   \textbf{Merge case:} If $D$ is in the merge case, first note that $B_t$ is uniquely determined by Lemma \ref{lem:clinging-unique}.  Then we have $k_t$ numbers with hue $\{t\}$ and $k_{B_t}$ numbers with hue $B_t$; in the merge, we now have $k_t+k_{B_t}=k_B$ numbers of hue $B$, and since we underlined $\min(B_t)$, the correct elements are underlined.  It is also clearly still mismatched.  Thus we have a well defined output.
\end{proof}

To prove that $\phihue$ is an involution, we need two key lemmas on how the number and hue chain decompositions change after applying a split or a merge. 

\begin{lemma}[Chain decompositions after a split]\label{rmk:aside1}
    Suppose $\phihue(D)=D'$ via the split case, splitting a block $B$ of color $C$ into $\{m\}$ and $B'=B\setminus \{m\}$, where $r=\min(B')$. Then the chain decompositions of $C$ in $D'$ can be obtained from that of $D$ as follows:
    \begin{itemize}
    \item[(Q)] In Situation Q, the number $r$ is inserted immediately after $q$ in the number chain decomposition, and the associated hue $B$ in the hue chain decomposition becomes $B',\{m\}$.
    \item[(\nQ)] In Situation \nQ, a new singleton number chain $(r)$, with associated hue chain $(\{m\})$, is inserted directly after the last number chain that starts with a number less than $r$. 
    \end{itemize}
    Either way, $q,r$ become the first numbers of hues $B'$ or $\{m\}$ in number chain order, and then all later occurrences of $B$ become either $B'$ or $\{m\}$ according to permission list order.
\end{lemma}

\begin{remark}
       Visually, the change in hue chain decomposition Situation Q looks something like: 
\begin{align*}
(a,b,c,\,{\bf q},\,m,e),\,(f,g),\,(h) &\to (a,\,b,\,c,\,{\bf q},\,{\bf r},\,m,\,e),\,(f,g),\,(h) \\
(X, Y, Z, {\bf B}, T, S), (A, B),(B)  &\to 
(X, Y, Z, {\bf B', \{m\}}, T, S), (A, B'),(\{m\})
\end{align*}
    and in Situation {\nQ}  like:
    \begin{align*}
    (a,b,c,d,e),(f,g),(h)&\to (a,b,c,d,e),({\bf r}),(f,g),(h)\\
        (X,Y, Z, T,S),(A, B),(B) &\to (X, Y, Z, T,S),{\bf(\{m\})},(A, B'),(\{m\})
    \end{align*}
    with $a < r < f$.
\end{remark}

\begin{proof}
We consider the two cases separately:
    
  \textbf{Situation Q.}  In this case, a (first) number $q$ of hue $B$ appears in $D$ in a number chain starting with $i<r$.   Among the new non-underlined letters, the changes that occurred in Steps 1--3 were that the $r$ was de-underlined and assigned hue $\{m\}$, and the hue of $q$ is changed from $B$ to $B'$.    Since $\min(B')=r$, the next number in the number chain is then our new de-underlined number $r$, which has hue $\{m\}$, and in particular the entries before $q$ in the number chain decomposition are unchanged.  Since we also have $m=\min(B)$, the rest of the number chain continues as before.  
  
  Finally, for Step 4, the remaining $B$ hues are changed to $B'$ or $\{m\}$ in permission list order, and note that these $B$ hues all occur at the ends of their chains since the number $m$ already occurred earlier.  So we change these to $B'$ first, then $m$ if we run out of $B'$ hues, and these are still at the ends of their chains because $r=\min(B')$ and $m$ have already occurred in number chain order in the new output.  Thus none of the other numbers or hues in the chain are affected, and we are done.

    \textbf{Situation \nQ.} Since $q$ does not exist, when we de-underline $r$ and consider the number chain decomposition of the intermediate hue labeling after Step 1, $r$ starts its own number chain, with hue $\{m\}$, after all numbers less than $r$ have already been listed.  Since $m<r$, the number $m$ has already been listed in the number chain and so the chain terminates as $(r)$.

    Since $m$ and $r$ has already been listed, any later $B$ hues in the number chain occur at the ends of their chains, and when we change them to $B'$ and $\{m\}$ according to Step 4 they remain at the ends of their chains and the rest of the number chain decomposition is unchanged.  This completes the proof.
\end{proof}

The opposite changes occur when we merge.

\begin{lemma}[Chain decompositions after a merge]\label{rmk:aside2}
    Suppose $\phihue(D')=D$ via the merge case, merging $\{t\}$ with a block $B_t$ (with $r=\min(B_t)$) into a block $B=\{t\}\cup B_t$ in color $C$. Then the chain decompositions of $C$ in $D$ can be obtained from that of $D'$ as follows:
\begin{itemize}
    \item[(Q)] In Situation Q, we have $q,r$ consecutively in a number chain with hues $B_t,\{t\}$, and we delete $r$ (and its hue $\{t\}$) and replace the hue $B_t$ of $q$ by $B$.
    \item[(\nQ)] In Situation \nQ, the chain $(r)$ terminates immediately, with hue chain $(\{t\})$; delete it.
\end{itemize}
In both cases, rename all other instances of hues $B_t$ and $\{t\}$ to $B$.
\end{lemma}

\begin{proof}
 Since $\{t\}$ has permission to merge, the number $r=\min(B_t)$ is the first number with hue $\{t\}$ in the number chain decomposition by condition (i) of permission. This $r$ becomes underlined after merging, so we must remove $r$ and its hue from the number and hue chain decompositions. 

\textbf{Situation Q:} In this situation, $r$ is preceded by a number $q$ with hue $B_t$ in a chain, and all numbers preceding $q$ in the number chain do not have hue $B_t$ or $\{t\}$, so the portion of the number chain preceding $q$ is unchanged after merging.   Then, $q$'s hue is changed to $B$, which has minimum element $t$.  Thus if number $t$ followed $r$ in its number chain in $D'$, then $t$ will now follow $q$ in $D$, so we have simply deleted $r$ and its hue without disturbing the rest of the chains.  

Any later $B_t$ or $\{t\}$ in hue chain order must have ended a chain since the numbers $r$ and $t$ were already listed, so changing them to $B$ results in $B$ hues that remain at the ends of their chain by the same reasoning.  Thus the rest of the chain is unchanged. 
 
 \textbf{Situation \nQ:} In this case $(r)$ is a singleton chain, and we remove it because it becomes underlined.  Moreover, since $(r)$ starts its own chain, the number $t<r$ must have preceded it in number chain order.  Thus, again, any later $B_t$ or $\{t\}$ in hue chain order must have ended its chain, so changing them to $B$ results in $B$ hues that remain at the ends of their chain, and we're done.  
\end{proof}

As another step towards showing that $\phihue$ is an involution, we now prove two technical lemmas about compatiblity of the split procedure with the conditions for permission to merge.

\begin{lemma}\label{lem:undo-split}
    Suppose $\phihue(D)=D'$ via the split case with block $B$, and $m=\min B$.   Then in $D'$, the hue $\{m\}$ has permission to merge with $B'=B\setminus \{m\}$.  
    
    Moreover, if $\{m\}$ is the smallest singleton hue with permission to merge in $D'$, then $\phihue(D')=D$.
\end{lemma}

\begin{proof}
    Since $\{m\}$ is the new hue of $r=\min(B')$ in $D'$ by Step 1 of the split procedure, $\{m\}$ wants to merge with $B'$ in $D'$.  By Lemma \ref{rmk:aside1} and the definition of $q$ for Situations Q and \nQ, the new number $r$ that appears in the number chain decomposition for $D'$ is the first number of hue $\{m\}$ in number chain order.  Thus condition (i) of permission is satisfied.

   For condition (ii), we consider Situation Q and Situation {\nQ} separately. In Situation Q, a number $q$ with hue $B'$ is labeled by Step 3, and directly precedes $r$ in the number chain of $D'$. In Step 4, we list the remaining numbers besides $q,r$ with hue $B$ in permission list order, and relabel them as $B'$ and then $\{m\}$.  In Situation \nQ, there is no $q$ and so we skip Step 3 and simply relabel the remaining numbers besides $r$ with hue $B$ in permission list order for Step 4.  Either way, the resulting hue order precisely corresponds to the condition (ii) for permission.
   
    Finally, if $\{m\}$ is the smallest singleton hue with permission, then $B'$ is the unique block it has permission to merge with by Lemma \ref{lem:clinging-unique}.  Moreover, there is no non-singleton block with minimum less than $m$ in $D'$, because $m$ was the smallest of any non-singleton block in $D$, so we do indeed merge $\{m\}$ with $B'$.  Merging $\{m\}$ with $B'$ re-underlines $r$ and restores all hues to the previous state in $D$, so $\phihue(D')=D$.
\end{proof}

\begin{lemma}\label{lem:undo-merge}
    Suppose $\phihue(D')=D$ via the merge case with $\{t\}$ merging with $B_{t}$. If $D$ is in the split case, then the block we split is the new block $B=\{t\}\cup B_{t}$ and $\phihue(D)=D'$.
\end{lemma}

\begin{proof}
    If $D$ is in the split case, we must split the block with the smallest minimum; this minimum is $t$ due to the merge definition. Thus the output $\phihue(D)$ is formed by splitting the new block $B$ back into $t$ and $B_{t}$.  
    
    Notice that $\phihue(D)$ agrees with $D'$ on its underlined letters and their hues by Step 2 of the split procedure, and the Step 1 rule of relabeling $r=\min(B_{t})$ as hue $\{t\}$ returns $r$ to its original label (because $t$ wanted to merge with $B_t$ in $D'$). 

    Notice also that $D$ is in Situation Q for the split procedure if and only if $D'$ and $\phihue(D)$ are both in Situation Q for permission, by Lemmas \ref{rmk:aside1} and \ref{rmk:aside2}.  Thus Step 3, if it occurs, also returns the number $q$ to its original hue.
    
    Finally, Step 4 of the split procedure restores the remaining hues to their permission list order in $D'$, because their original order had all $B_t$'s preceding all $\{t\}$'s by condition (ii) of permission to merge. Thus indeed $\phihue(D)=D'$.
\end{proof}

The next two propositions combined show that $\phihue$ is an involution.  Note that here is where we will use the chosen convention that the exponents $\mathbf{k}$ are weakly increasing on each color; this is actually an essential part of the setup in order for the map to be an involution.

\begin{prop}\label{prop:split}
Suppose $\phihue(D)=D'$ via the split case. Then $\phihue(D') = D$.
\end{prop}

\begin{proof}
  Let $B$ be the block in $D$ that is split, with $m = \min(B)$ and $B'=B\setminus \{m\}$.   By Lemma \ref{lem:undo-split}, it remains to show that there is no $t<m$ that has permission to merge with some $B_{t}$ in $D'$.

  Assume for contradiction that some $t<m$ has permission to merge with a block $B_{t}$.  We consider several possibilities for $\{t\},B_{t}$ in terms of whether each is equal to either $\{m\}$ or $B'$.  If neither $\{t\}$ nor $B_{t}$ is equal to either $\{m\}$ or $B'$, then by Lemma \ref{rmk:aside1}, their positions in hue chain decomposition order are unchanged from $D$ to $D'$, and so $\{t\}$ also had permission to merge with $B_{t}$ in $D$ and $t<m$, contradicting the fact that $D$ is in the split case.   If $B_{t}=B'$, then $\{t\}$ and $\{m\}$ both want to merge with $B_{t}$, a contradiction.  If $\{t\}$ is equal to either $\{m\}$ or $B'$ then we have a contradiction because we assumed $t<m$.

  Thus, the only remaining possibility is that $B_{t}=\{m\}$.  Then $\{t\}$ wants to merge with $\{m\}$ and is the hue of number $m$.  Therefore $\{t\}$ wants to merge with $B$ in $D$.  Since $\{t\}$ has permission from $\{m\}$, the $\{t\}$ labeling $m$ is the first $\{t\}$ in hue chain order in $D'$, and so it also is first in $D$ by Lemma \ref{rmk:aside1}.  Thus $\{t\}$ satisfies condition (i) for permission to merge with $B$ in $D$.

  We will show that $\{t\}$ and $B$ also satisfy condition (ii) to merge in $D$, which is a contradiction because $D$ is in the split case.  If $k_m=1$, then $k_t=1$ as well by our assumed ordering of $\mathbf{k}$, and so there are no $\{t\}$ hues besides the first.  Thus condition (ii) for $\{t\}$ and $B$ is vacuously satisfied.  We therefore may assume that $k_m>1$, which means there is some number with hue $\{m\}$ in $D'$ besides the first in hue chain order.

  We now consider two cases, based on whether $\{t\},\{m\}$ are in Situation Q or NQ.

\textbf{Situation Q.}  If $\{t\},\{m\}$ are in Situation Q for merging, then the first $\{m\}$ and $\{t\}$ hue occur consecutively in a hue chain.  Thus the first number of hue $\{m\}$ does not form a singleton chain, and so $\{m\},B'$ are also in Situation Q, and the first $B'$ occurs just before the first $\{m\}$ in hue chain order.  Thus we have numbers $q,r,m$ consecutively in a hue chain with hues $B',\{m\},\{t\}$, and these are the first instances of each of these hues in hue chain order in $D'$.  In $D$, this hue chain correspondingly has consecutive entries $q,m$ with hues $B,\{t\}$ by Lemma \ref{rmk:aside1}.

Let $S$ be the set of all remaining numbers (besides the first three in number chain order) with hues $B'$, $\{m\}$, or $\{t\}$, and consider the permission list order of $S$.  Note that the permission list order for the remaining numbers of hue $B',\{m\}$ is an ordered sub-list of the full permission list, and same for $\{m\},\{t\}$.  Since $k_m>1$, there is at least one number with hue $\{m\}$ in $S$, and so by transitivity all the numbers with hue $\{t\}$ in the list occur after those with hue $B'$ as well.  Thus, in $D$, before splitting, the numbers with hue $\{t\}$ in the permission list for $\{t\}$ and $B$ occur after those with hue $B$, again by Lemma \ref{rmk:aside1}. It follows that $\{t\}$ has permission to merge with $B$ in $D$, a contradiction.

\textbf{Situation NQ.}  If $\{t\},\{m\}$ are in Situation NQ, then $(m)$ is a singleton number chain with hue chain $(\{t\})$, and this is the first $\{t\}$ in hue chain order.  Note that this $(\{t\})$ occurs before the first $\{m\}$ in hue chain order, because if a number with hue $\{m\}$ appeared earlier, it would point to number $m$ in the ordering next.  

Now, $\{m\},B'$ can either be in Situation Q or Situation NQ, and correspondingly the hue chain decomposition in $D'$ can look like either:
$$\cdots (\{t\})\cdots (\ldots,B',\{m\})\cdots (\ldots, B')\cdots (\ldots,B')\cdots(\ldots,\{m\})\cdots (\ldots,\{t\})\cdots $$
or
$$\cdots (\{t\})\cdots (\{m\})\cdots (\ldots, B')\cdots (\ldots,B')\cdots(\ldots,\{m\})\cdots (\ldots,\{t\})\cdots $$  
Note that since $\{t\},\{m\}$ satisfy condition (ii) for permission, all $\{t\}$ hues besides the first occur after all $\{m\}$ hues in permission list order, including the first $\{m\}$ hue and including necessarily at least one $\{m\}$ hue elsewhere in the decoration (since $k_m>1$).  But these other $\{m\}$ hues occur after all extra $B'$ hues in permission list order, which all occur to the right of the first $\{m\}$ shown above.  Thus all extra $\{t\}$ hues occur after any $B'$ or $\{m\}$ that were labeled $B$ in $D$, and so $\{t\},B$ satisfy condition (ii) for permission in $D$, again a contradiction.
\end{proof}

\begin{prop}\label{prop:merge}
    Suppose  $\phihue(D')=D$ via the merge case.  Then $\phihue(D)=D'$. 
\end{prop}

\begin{proof}
  Suppose $D'$ was in the merge case with $t'$ and $B_{t'}$, and set $B= \{t'\}\cup B_{t'}$. By Lemma \ref{lem:undo-merge}, it suffices to show that $D$ is in the split case.  Assume for contradiction that $D$ is in the merge case, with some pair $t,B_t$ with $t<t'$.  If $\{t\}$ has permission to merge with $B_t$ in $D$ and both are not equal to block $B$, then by Lemma \ref{rmk:aside2} all positions of hues $\{t\}$ and $B_t$ are unchanged in hue chain order between $D$ and $D'$.  Thus they were a pair with permission in $D'$ as well, a contradiction to the minimality of $t'$ for the merge case.  Note that $\{t\}$ cannot be $B$ either since $B$ is not a singleton.

    So, suppose $B_t=B$.  Then $\{t\}$ wants to merge with $B$ in $D$, so it is the hue of $t'=\min(B)$, and we wish to show it had permission to merge with $\{t'\}$ in $D'$, contradicting the minimality of $t'$ in $D'$.  Since $\{t\}$ has permission to merge with $B$ in $D$, the number $t'$ is the first number of hue $\{t\}$ in number chain order for $D$, and by Lemma \ref{rmk:aside2}, it is also the first $\{t\}$ in hue chain order for $D'$.  Thus condition (i) for permission is satisfied.

    We consider the two cases of whether the pair $\{t\},B$ in $D$ is in Situation Q or \nQ\ separately to show that condition (ii) is satisfied for $\{t\}$ and $\{t'\}$.

    \textbf{Situation Q.}  In this case, there is a (first) number $q$ with hue $B$ preceding $t'$ (with hue $\{t\}$) in number chain order.  Since $q$ does not form its own singleton chain, $D'$ must have been in Situation Q for permission as well, and so in $D'$ we have $q,r,t'$ consecutively in a number chain, with hues $B_{t'},\{t'\},\{t\}$, and these are the first instances of each of these three hues in number chain order.  

    Among the remaining hues $B,\{t\}$ in permission list order in $D$ (after numbers $q,t'$), those of hue $B$ precede those of hue $\{t\}$ by condition (ii) for permission.  By Lemma \ref{rmk:aside2}, in $D'$, those of hue $B$ came from relabeling the hues $B_{t'}$ and $\{t'\}$ in those positions, and so indeed, in permission list order, those of hue $\{t'\}$ come before those of hue $\{t\}$ in $D'$ as well.  Thus condition (ii) is satisfied for $\{t\},\{t'\}$ in $D'$.

    \textbf{Situation \nQ.}  In this case, $(t')$ is its own number chain, with hue chain $(\{t\})$.  
    Thus it is its own number chain in $D'$ as well by Lemma \ref{rmk:aside2}. Writing $r = \min(B_{t'})$, since $r$ has hue $\{t'\}$, it would point to number $t'$ if $r$ occurred before $t'$ in number chain order, a contradiction.  So $r$ occurs later in number chain order than $t'$ in $D'$.  Thus the hue chain decomposition of $D$ has one of the following two forms:
    \begin{equation*}
        \ldots(\{t\})\ldots (\{t'\})\ldots (\cdots B_{t'})\ldots (\cdots B_{t'})\ldots (\cdots \{t'\})\ldots
    \end{equation*} 
    or
    \begin{equation*}
        \ldots(\{t\})\ldots (\cdots B_{t'},\{t'\}\cdots)\ldots (\cdots B_{t'})\ldots (\cdots B_{t'})\ldots (\cdots \{t'\})\ldots
    \end{equation*} 
We claim that, excluding the first $\{t\}$ hue, all $\{t'\}$ hues precede all $\{t\}$ hues.  

We compare with the corresponding chains in $D$. By Lemma \ref{rmk:aside2}, the hue chain decomposition in $D$ respectively looks like either:
    $$\ldots(\{t\})\ldots ({\color{red}\times})\ldots (\cdots B)\ldots (\cdots B)\ldots (\cdots B)\ldots $$ or
    $$\ldots(\{t\})\ldots (\cdots B, {\color{red}\times}\cdots )\ldots (\cdots B)\ldots (\cdots B)\ldots (\cdots B)\ldots $$
where the ${\color{red} \times}$ indicates that we deleted that entry.

Note that any other $\{t\}$ hues must come after all of these later $B$ hues in permission list order because $\{t\}$ has permission from $B$, and so in particular they occur after the deleted $\{t'\}$ hue in $D'$ as well, and after all $\{t'\}$ hues in permission list order in $D'$.  Thus condition (ii) is satisfied in this case as well. 
\end{proof}

Propositions \ref{prop:split} and \ref{prop:merge} above combine to give our main result of this section.

\begin{thm}
    The map $\varphi^{\mathrm{hue}}$ is a sign reversing involution on $\mathcal{MD}(\mathbf{k})$.
\end{thm}

We now describe the fixed points of the involution, which all have positive sign.

\begin{defn} \label{def:fixedpoint-decoration}
    A mismatched decoration $D$ is a \textbf{fixed point} if it contains no $m$ or $t$ as in Definition \ref{def:involution-2}.  That is, its hue partition is the all-singleton partition, and no hue has permission to merge with any other.
\end{defn}

It is clear that these are precisely the fixed points of $\varphi^{\mathrm{hue}}$. Note also that a fixed point has no remaining underlined numbers (no unboxed underlined numbers since it is mismatched, and no boxed underlined numbers since $\mathcal{P}_D$ is all singletons). A fixed point decoration is thus equivalent to the data of a mapping from $\{1, \ldots, n\} \setminus \{1, r_1+1, r_1+r_2+1\}$ to the boxed numbers, each considered as a singleton hue, satisfying the conditions above.

\begin{example} \label{ex:special-fixed-points}
    As a special case of interest, consider a mismatched decoration $D$ whose hue partition is all singletons. If every boxed number has either its own hue, or a hue that does not match its color, $D$ is automatically a fixed point. (In particular, the number $r$ in Definition \ref{def:fixedpoint-decoration} cannot exist.)
\end{example}

\begin{example}
  Going back to the product in Example \ref{ex:decoration},  $\int \psi_1^2\psi_2^3\psi_3^5\cdot \psi_7^4\cdot \psi_{18}\psi_{19}$, below is an example of a fixed point.
\begin{center}
    \begin{tabular}{cccccccc}
    $R$ & & $\boxed{\cancel{\phantom{.}{\color{red}1}\phantom{.}}}$ & $\boxed{\phantom{.}{\color{red}2}_{R_2}\phantom{.}}$ & $\boxed{{\phantom{.}{\color{red}3}_{R_{1}}}\phantom{.}}$ & ${\color{red}4}_{V_{18}}$ & ${\color{red}5}_{A_{7}}$ & ${\color{red}6}_{A_{7}}$ \\
    
    $A$ & & $\boxed{\cancel{\phantom{.}\color{blue}7\phantom{.}}}$ & ${\color{blue}8}_{R_{3}}$ & ${\color{blue}9}_{R_2}$ & ${\color{blue}10}_{R_{3}}$ & ${\color{blue}11}_{R_1}$ & \\
    
 $G$ & &    $\cancel{{\color{mygreen}12}}$ & ${\color{mygreen}13}_{R_2}$ & ${\color{mygreen}14}_{A_7}$ & ${\color{mygreen}15}_{R_{3}}$ & & \\
    
    $Y$ & & ${\color{y1}16}_{R_{3}}$ & ${\color{y1}17}_{R_{3}}$ & & & & \\
    
 $V$ & &  $\boxed{{\color{y2}18}_{A_{7}}}$ & $\boxed{{\color{y2}19}_{V_{19}}}$ & & & &
\end{tabular}
\end{center}

Notice that $\{1\}$ wants to merge with $\{3\}$, but it cannot, because there is an $R_1$ before an $R_3$ on the non-red letters in ascending order.  
\end{example}

Combining the sign reversing involution of Definition \ref{def:involution-2} with Lemma \ref{lem:alternating-2}, we obtain the following theorem.

\begin{thm}\label{thm:main2}
    The intersection product \eqref{eqn:psi-product} is equal to the number of fixed point decorations with parameters given by the exponents. In particular, the product is nonzero if and only if a fixed point decoration exists.
\end{thm}

\begin{proof}
    The map $\varphi^{\mathrm{hue}}$ cancels the remaining terms, and since none of the fixed point decorations have any underlined numbers, they all contribute $+1$ to the sum.
\end{proof}

\section{Applications} \label{sec:examples}

We now give a number of special cases and consequences of Theorem \ref{thm:main2}.

\subsection{A concrete calculation}

Consider 10 marked points in three colors with $r_1 = 5, r_2 = 4, r_3 = 1$. We calculate $\int \psi_1^2 \cdot \psi_6^2\psi_7^3$. Some of the hue assignments are determined, and the others are denoted below with question marks:

 \begin{center}
    \begin{tabular}{cccccccc}
    $R$ &
    & $\boxed{\cancel{\phantom{.}{\color{red}1}\phantom{.}}}$
    & ${\color{red}2}_{A_{?}}$
    & ${\color{red}3}_{A_{?}}$
    & ${\color{red}4}_{A_{?}}$
    & ${\color{red}5}_{A_{?}}$ \\
    
    $A$ &
    & $\boxed{\cancel{\phantom{.}\color{blue}6\phantom{.}}}$
    & $\boxed{{\phantom{.}{\color{blue}7}_{A_{?}}\phantom{.}}}$
    & ${\color{blue}8}_{R_1}$
    & ${\color{blue}9}_{R_1}$ \\

    $G$ &
    & $\cancel{{\color{mygreen}10}}$
    
\end{tabular}
\end{center}
It remains to determine the assignments of hues $A_6$ and $A_7$ to $2, 3, 4, 5$ and $7$.

If $7$ has hue $A_7$, the decoration is a fixed point by Example \ref{ex:special-fixed-points}, and there are $\binom{4}{2}=6$ ways to choose which of $2,3,4,5$ has hue $A_7$ and which of the other two have hue $A_6$.  If instead $7$ has hue $A_6$, then by Definition \ref{def:involution-2} with $t=6$, the decoration is non-fixed if and only if all other occurrences of hue $A_7$ precede all occurrences of $A_6$ (excluding at $7$ itself).  Thus there are $4-1=3$ fixed points in this case, for a total of $9$ fixed point decorations.  It follows that 
\[
\int_{\overline{M}_{0, [5, 4, 1]}} \psi_1^2 \cdot \psi_6^2 \psi_7^3 = 9.
\]

\subsection{Losev-Manin case} \label{sec:losev-manin}

We apply Theorem \ref{thm:main2} to recover the computation of all integrals \eqref{eqn:psi-product} over the Losev-Manin space $\mathrm{LM}_{2\vert n}$, which we view as having three colors -- red, azure, and green -- with one red marked point $0$, one green marked point $\infty$, and $n$ azure marked points $1,2,\ldots,n$. 

\begin{example} \label{ex:Losev-Manin-fixed-points}
    Suppose we wish to compute $\int \psi_0\cdot \psi_1^2\psi_2\cdot \psi_\infty^2$ on $\mathrm{LM}_{2|7}$. Then the fixed point decorations would have to label a diagram that looks like: 

\begin{center}
    \begin{tabular}{ccccccccc}
        $R$ & & $\boxed{\cancel{\phantom{.}0\phantom{.}}}$ \\
        $A$ & & $\boxed{\cancel{\phantom{.}1\phantom{.}}}$ & $\boxed{2}$ & 3 & 4 & 5 & 6 & 7 \\
        $G$ & & $\boxed{\cancel{\infty}}$
    \end{tabular}
\end{center}
By counting, there is not enough room to place the three azure colors without assigning an unboxed azure number an azure hue. In particular, the set $\mathcal{MD}(\mathbf{k})$ of mismatched decorations is empty. Similarly, $\mathcal{MD}(\mathbf{k})$ is empty, and the integral vanishes, whenever there is a $\psi$ class on the azure numbers. In fact, it is known that those $\psi$ classes are zero in $A^1(\mathrm{LM}_{2|7})$ (see \cite{cavalieri2016}).

    In contrast, if we only have $\psi$ classes on the marked points $0$ and $\infty$, we can place the colors on the azure numbers in any manner to obtain a fixed point, such as the following fixed point decoration for $\int \psi_0^4\psi_\infty^2$:

\begin{center}
    \begin{tabular}{ccccccccc}
        $R$ & & $\boxed{\cancel{\phantom{.}0\phantom{.}}}$ \\
        $A$ & & $\cancel{\phantom{.}1\phantom{.}}$ & $2_R$ & $3_G$ & $4_G$ & $5_R$ & $6_R$ & $7_G$ \\
        $G$ & & $\boxed{\cancel{\infty}}$
    \end{tabular}
\end{center}
\end{example}
The above two examples generalize as follows.

\begin{prop}[{{\cite[Exercise 52]{cavalieri2016}}}]\label{prop:losev-manin}
    The product $\int \psi_0^a\cdot \psi_1^{k_1}\cdots \psi_\ell^{k_\ell} \cdot \psi_\infty^b$ on $\mathrm{LM}_{2|n}$ (with $a+b+\sum k_i=n-1$) is equal to $0$ if any $k_i>0$, and if all $k_i=0$ then we have $$\int_{\mathrm{LM}_{2|n}} \psi_0^a\psi_\infty^b=\binom{n-1}{a}.$$ 

\end{prop}

\subsection{Compatibility with color merge}\label{subsec:compatibility-with-color-merge}

Whenever $m \geq 3$, there is a reduction map corresponding to merging colors,
\[
\rho : \Mbar_{0, [r_1, \ldots, r_m, r_{m+1}]} \to
\Mbar_{0, [r_1, \ldots, r_m + r_{m+1}]}.
\]
If $k_i = 0$ for all $i$ in the merged colors, we have
\[
\int_{\Mbar_{0, [r_1, \ldots, r_m, r_{m+1}]}} \psi_1^{k_1} \cdots \psi_n^{k_n}
=
\int_{\Mbar_{0, [r_1, \ldots, r_m + r_{m+1}]}}
\psi_1^{k_1} \cdots \psi_n^{k_n}
\]
by Corollary \ref{cor:color-clumping}. This identity is also immediate using decorations: there are no hues of colors $C^{(m)}$ or $C^{(m+1)}$ to assign to other numbers, and it makes no difference to the mismatching condition or the actions of $\underline{\varphi}$ and $\phihue$ if the numbers from $C^{(m)}$ and $C^{(m+1)}$ are placed in two rows or just one.

\subsection{One \texorpdfstring{$\psi$}{} class in each color and reduction to \texorpdfstring{$\Mbar_{0,n}$}{}}

Suppose we have a $\psi$ class product with at most one $\psi$ class with a nonzero exponent in each color.
Then the mismatched decorations must all have one singleton hue of each color, and so have no $m$ or $t$ as in Definition \ref{def:involution-2}, and so all mismatched decorations are automatically fixed point decorations.  Note that this means that we attain both the lower and upper bounds of Corollaries \ref{cor:lower-bound} and \ref{cor:upper-bound} in Section \ref{sec:cerberus} below, which coincide in this case.

As a special case, we can recover the known $\psi$ class intersection formula on $\Mbar_{0,n}$, namely, $$\int_{\Mbar_{0,n}}\psi_1^{k_1}\cdots \psi_n^{k_n} = \binom{n-3}{k_1, \ldots, k_n}$$ with $\sum k_i=n-3$.  In the language of decorations, there is one point of each color, and any number for which $k_i > 0$ has no unboxed numbers in its row; thus every decoration is mismatched, and as above they are all fixed points.  These are enumerated by distributing the $\psi$ class colors over the $n-3$ numbers that are not crossed out, which recovers the formula above.

\subsection{All \texorpdfstring{$\psi$}{} classes in a single color and heavy-light spaces}
\label{subsec:single-color-heavy-light}

Consider the case where all the $\psi$ classes in the intersection product \eqref{eqn:psi-product} are of a single color, without loss of generality $C^{(1)}$. As in Section \ref{subsec:compatibility-with-color-merge}, the distribution of the remaining colors then makes no difference to the mismatching condition or the actions of $\underline{\varphi}$ and $\varphi^{\mathrm{hue}}$. There is thus no loss of generality in taking the space to have three colors $\overline{M}_{[c_1, c_2, 1]}$ or to be a heavy-light space $\overline{M}_{[c_1, 1, \ldots, 1]}$.

If some number in color $C^{(1)}$ is unboxed, then by counting, there are no mismatched decorations and the sum yields $0$ after applying $\underline{\varphi}$. If instead every number in color $C^{(1)}$ is boxed, the mismatched condition is automatically satisfied; in this case, $\underline{\varphi}$ does nothing, and any cancellations are carried out by $\varphi^{\mathrm{hue}}$.

In the latter case, it is straightforward to see that the integral is always nonzero, by assigning each number $r$ in row $1$ itself as its hue (or a hue $C_{\{t\}}$ with $r \leq t$) and assigning all other hues arbitrarily. By Example \ref{ex:special-fixed-points}, every such assignment is a fixed point, leading to bounds
\[
\binom{n-3-r_1}{k_1-1, k_2-1, \ldots, k_{r_1}-1}
\leq
\int_{\overline{M}_{[r_1, \ldots]}} \psi_1^{k_1} \cdots \psi_{r_1}^{k_{r_1}}
\leq
\binom{n-3}{k_1, \ldots, k_r}.
\]
We do not know of any further simplifications to the count of fixed-point decorations in this case, since the requirement that no hue have permission to merge is delicate to count. It is perhaps surprising that integrals of `light' $\psi$ classes on heavy-light spaces exhibit this complexity.

\subsection{Border strip case}
The product
    \[        \int_{\overline{M}_{[r,1^{r+3}]}} (\psi_1 \cdots \psi_r)^2 = 1, 5, 61, 1379, 49946, \ldots \text{ for } r = 1, 2, \ldots \]
is known to be the number of tilings of a $2r \times r$ rectangle by length-$r$ border strips (OEIS \href{https://oeis.org/A115047}{A115047}). We initially observed this coincidence of numbers via fixed point decorations by computer calculation (up to $r=5$), and we thank Ian Cavey and Deniz Genlik for explaining this to us, as follows: by \cite{pandharipande_kappa}, we may reformulate the integral as $\int_{\overline{M}_{0, r+3}} \kappa_1^r$, which is equivalent to the normalized Weil--Petersson volume of $\overline{M}_{0,r+3}$ by \cite{kaufmann1996higher}, and which was first computed by Zograf
\cite{Zograf}. Then, by \cite{alexandersson-jordan}, this integral equals the number of border strip decompositions above.

It would be interesting to find a bijection between border strip tilings and fixed point decorations in this case.

\section{Condition for the existence of fixed points}
\label{sec:cerberus}

We now give necessary and sufficient conditions for the product \eqref{eqn:psi-product} of $\psi$ classes on a multicolored space to be nonzero. As above, let $\ell_j$ be the number of distinct $\psi$ classes of color $C^{(j)}$ appearing in the product, and let $k_{C^{(j)}}$ be the total of the exponents of color $C^{(j)}$, with $k_{C^{(1)}}+\cdots+k_{C^{(m)}}=n-3$. 

\begin{thm}\label{thm:nonzero}
The intersection product \eqref{eqn:psi-product} is nonzero if and only if the inequalities 
\begin{equation}\label{eqn:inequalities}
k_{C^{(j)}}\le \ell_j+n -3 - r_j,
\end{equation}
hold for all $j$ such that $k_{C^{(j)}} \ne 0$.
\end{thm}

\begin{remark}[Inequalities for all $j$ and Losev-Manin space]
    If $k_{C^{(j)}} = 0$, then necessarily $\ell_j = 0$, and so the inequality \eqref{eqn:inequalities} becomes
    \[
    0 \leq n-3-r_j,
    \]
    which is nearly always satisfied, so there is (almost) no harm in taking \eqref{eqn:inequalities} for all $j$. The exceptional case is Losev-Manin space, where $r_j = n-2$ for the `light' color and the inequality \eqref{eqn:inequalities} is $k_{C^{(j)}} \leq \ell_j - 1$ and is never satisfied. But fixed point decorations do in fact exist if (and only if) $k_{C^{(j)}} = 0 = \ell_j$; see Example \ref{ex:Losev-Manin-fixed-points}. A reformulation that holds in all cases is: for all $j$,
    \begin{equation}\label{eqn:inequalities-valid-for-LM}
k_{C^{(j)}}\le \max\left(\ell_j+n -3 - r_j, 0\right).
\end{equation}
\end{remark}

These inequalities are closely related to the existence of certain simplified assignments of colors (rather than hues) to numbers, similar to fixed point decorations.

\begin{defn}
    A {\bf mismatched coloring} of type $\mathbf{k}$ is an assignment of colors to the non-crossed-out numbers, such that a total of $k_{C^{(j)}}$ numbers are assigned the color $C^{(j)}$, for each $j$, and no unboxed number is assigned its own color.
\end{defn}

\begin{lemma} \label{lem:surjection-colorings}
    There is a surjection from fixed point decorations to mismatched colorings. In particular, a fixed point decoration exists if and only if a mismatched coloring exists.
\end{lemma}
\begin{cor} \label{cor:lower-bound}
    The integral \eqref{eqn:psi-product} is bounded below by the number of mismatched colorings.
\end{cor}

\begin{proof}[Proof of Lemma \ref{lem:surjection-colorings}]
    Let $D$ be a fixed point decoration. Forgetting the distinctions between hues and remembering only the underlying colors gives a mismatched coloring.

    Conversely, consider an arbitrary mismatched coloring. We define a decoration by specializing the color assignments into hues, as follows. We differentiate the colors into singleton hues and, whenever a (boxed) number $i$ was assigned its own color, we now assign it its own hue. We specialize all other colors to hues arbitrarily, such that each hue $C_i^{(j)}$ occurs $k_i^{(j)}$ times. Any such assignment yields a fixed point decoration, since every boxed number is assigned either its own hue, or a hue that does not match its color (see Example \ref{ex:special-fixed-points}).
\end{proof}

\begin{proof}[Proof of Theorem \ref{thm:nonzero}]
    By Lemma \ref{lem:surjection-colorings}, a fixed point decoration exists if and only if a mismatched coloring exists. To characterize the existence of a mismatched coloring, we use Hall's Marriage Lemma. We consider the bipartite graph with the non-crossed out numbers as its left set and the available colors (with multiplicity) as its right set, and where each number is connected by an edge to all the colors it can be assigned: any color if it is boxed, any color besides itself if it is unboxed. Hall's Marriage Lemma states that there is a perfect matching, and hence a mismatched coloring, if and only if the neighborhood of any set $S$ on the right has at least as many elements as $S$. 
    
    If $S$ contains vertices of two or more colors, its neighborhood is the entire left-hand side of the graph, so the inequality is trivially satisfied. Suppose now that $S$ has vertices of a single color $C^{(j)}$ appearing with $k_{C^{(j)}} \ne 0$. We have $|S| \leq k_{C^{(j)}}$. For the first three colors (those with a crossed-out number), the neighborhood has size $(\ell_j - 1) + (n - 2 - r_j)$. For the other colors, the neighborhood has size $\ell_j + (n-3-r_j)$ (i.e. the same size). Thus, a matching exists if and only if, whenever $k_{C^{(j)}} \ne 0$, the inequality $k_{C^{(j)}} \leq \ell_j + n - 3 - r_j$ holds.
\end{proof}

\begin{cor}\label{cor:upper-bound}
    Let $G$ be the bipartite graph from the proof of Theorem \ref{thm:nonzero}. Then the integral \eqref{eqn:psi-product} is bounded above by the number of matchings on $G$.
\end{cor}
\begin{proof}
    We may identify the vertices of the right-hand set of $G$ with the boxed, non-crossed-out numbers from $\{1, \ldots, n\}$, i.e. as singleton hues. We may thus view the set of fixed point decorations as a subset of the set of matchings on $G$, namely those satisfying the final condition of Definition \ref{def:fixedpoint-decoration}. 
\end{proof}

Note that, in our bipartite graph, if we treat vertices of the same color as indistinguishable in $G$ and consider the corresponding matchings equivalent, we instead obtain the lower bound of Corollary \ref{cor:lower-bound}.  In other words, the number of fixed points is bounded below by the number of mismatched \textit{colorings}, and bounded above by the number of mismatched \textit{decorations} with all singleton blocks.

\begin{example}
    Consider an intersection product \eqref{eqn:psi-product} in which $\ell_j \leq 1$ for every $j$, that is, at most one $\psi$ class occurs in each color. Then the lower and upper bounds of Corollaries \ref{cor:lower-bound} and \ref{cor:upper-bound} coincide, so the integral is the number of mismatched colorings (equivalently, matchings of $G$).
\end{example}

\begin{remark}[Comparison to Kapranov degrees and cross-ratio degrees] \label{rmk:comparison-to-BELL-SIL}
    It is interesting to compare Theorem \ref{thm:nonzero} and Corollaries \ref{cor:lower-bound}--\ref{cor:upper-bound} to \cite[Theorem 1.1]{silversmith2021crossratio} and \cite[Theorem C]{brakensiek}, which give upper bounds for certain products of pullbacks of $\psi$ classes in terms of matchings of an associated bipartite graph, just as we find both upper and lower bounds in terms of matchings. Our inequalities \eqref{eqn:inequalities} are likewise similar in spirit to the Cerberus condition \cite[Theorem A]{brakensiek} for positivity of products of pullbacks of $\psi$ classes on $\overline{M}_{0, n}$.
\end{remark}

\bibliography{refs}
\bibliographystyle{amsalpha}

\end{document}